\documentclass[12pt,a4paper]{amsart}
\usepackage{latexsym}
\usepackage{graphics} 
\usepackage{epsfig} 
\usepackage{amsmath} 
\usepackage{amsfonts, amssymb} 
\usepackage{amsthm,enumerate}
\usepackage{color}
\usepackage{mathrsfs}
\usepackage{amscd}      

\theoremstyle{definition}
\newtheorem{Def}{Definition}[subsection]
\newtheorem{Eg}{Example}[section]
\newtheorem{Rm}{Remark}[section]

\theoremstyle{plain}
\newtheorem{Prop}[Def]{Proposition}
\newtheorem{Lem}[Def]{Lemma}
\newtheorem{Thm}[Def]{Theorem}
\newtheorem{Cor}[Def]{Corollary}

\numberwithin{equation}{section}

\newcommand{\authorfootnotesA}{\renewcommand\thefootnote{$\flat$}}%
\newcommand{\authorfootnotesB}{\renewcommand\thefootnote{$\sharp$}}%
\newcommand{\authorfootnotesC}{\renewcommand\thefootnote{$\diamond$}}%

\renewcommand{\thefootnote}{\fnsymbol{footnote}}

\begin{document}
\begin{center}
	\LARGE 
	Convergence Implications via Dual Flow Method \par \bigskip

	\normalsize
	\authorfootnotesA
	Takafumi Amaba\footnote{The first author was supported by
	JSPS KAKENHI Grant Number 15K17562.}\textsuperscript{1},
	\authorfootnotesB
	Dai Taguchi\textsuperscript{2}
	and
	\authorfootnotesC
	G\^o Y\^uki\textsuperscript{3}
	\par \bigskip

	\textsuperscript{1,}\textsuperscript{2,}\textsuperscript{3}Ritsumeikan University,
	1-1-1 Nojihigashi, Kusatsu, Shiga, 525-8577, Japan\par \bigskip

	\email{(T. Amaba) fm-amaba@fc.ritsumei.ac.jp}
	\email{(D. Taguchi) dai.taguchi.dai@gmail.com}
	\email{(G. Y\^uki) goyuki@fc.ritsumei.ac.jp}

	\today

\begin{quote}{\small {\bf Abstract.}
	Given a one-dimensional stochastic differential equation,
	one can associate to this equation
	a stochastic flow
	on $[0,+\infty )$,
	which has an absorbing barrier at zero.
	Then one can define its dual stochastic flow.
	In \cite{AW}, Akahori and Watanabe showed that
	its one-point motion solves
	a corresponding stochastic differential equation
	of Skorokhod-type.
	In this paper, we consider a discrete-time stochastic-flow
	which approximates the original stochastic flow.
	We show that under some assumptions, one-point motions of its dual flow
	also approximates the corresponding reflecting diffusion.
	We investigate the relation between them
	in weak and strong approximation sense.
}
	\end{quote}

\end{center}

\footnote[0]{ 
2010 \textit{Mathematics Subject Classification}.
Primary 60H10; 
Secondary
60J60, 
60J25. 
}

\footnote[0]{ 
\textit{Key words and phrases}.
Stochastic-flow on $[0, +\infty )$,
Dual stochastic flow,
Siegmund's duality,
Absorbing diffusion,
Reflecting diffusion,
Euler-Maruyama approximation.
}

\section{Introduction}
\noindent
\underline{{\it Main result.}}

We consider an absorbing stochastic flow 
$\{ X_{s,t} \}_{s \leq t}$ 
associated with the following stochastic differential equation
(the precise definition and assumptions will be given in the subsection \ref{SFAB}):
\begin{equation}
\label{Intr-Abs} 
\mathrm{d} X_{t}
=
\sigma ( X_{t} ) \mathrm{d}w(t)
+
b ( X_{t} ) \mathrm{d}t ,
\quad
X_{0} > 0.
\end{equation}
One can define its {\it dual stochastic flow}
$\{ X_{s,t}^{*} \}_{s \leq t}$
by
$
X_{s,t}^{*} := X_{-t,-s}^{-1}
$
(the right-continuous inverse).
In \cite{AW}, Akahori and Watanabe showed that
the one-point motion of
$\{ X_{s,t}^{*} \}_{s \leq t}$
solves a corresponding Skorokhod-type stochastic differential equation. 
This fact will be recalled in Theorem \ref{Principle}.

In this paper, we introduce a discrete-time stochastic-flow
$\{ X_{k,l}^{n} \}_{k \leq l}$, $n \in \mathbb{N}$
constructed by the Euler-Maruyama type approximation for $\{ X_{s,t} \}_{s \leq t}$
(see Definition \ref{EM-AbsFl}).
Then the result of Akahori and Watanabe suggests that
the sequence of one-point motions of the dual flow
$\{ X_{k,l}^{n*} \}_{k \leq l}$
would be a natural candidate of the approximation for
solution to the Skorokhod-type stochastic differential equation.
Hence we are motivated to investigate
the relation between
convergences
$
\{ X_{k,l}^{n} \}_{k \leq l}
\to
\{ X_{s,t} \}_{s \leq t}
$
and
$
\{ X_{k,l}^{n*} \}_{k \leq l}
\to
\{ X_{s,t}^{*} \}_{s \leq t}
$.

The next theorem says that
to know the rate of convergence for
$
\vert
\mathbf{E} [ f(X_{0,T}^{*}(x)) ]
-
\mathbf{E} [ f(X_{0,n}^{n*}(x)) ]
\vert
$,
it suffices to investigate the estimate for
$
\vert
\mathbf{P} ( X_{T}(y) > x )
-
\mathbf{P} ( X_{T}^{n}(y) > x )
\vert
$,
{\it with clarifying the dependence on the initial point $y>0$}.

\begin{Thm} 
\label{intro_main1} 
Assume the conditions (i)--(iv) in the subsection \ref{SFAB}.
Let $f:[0, +\infty ) \to \mathbb{R}$ be a
differentiable function with $f(0)=0$
and with compact support.
For $x > 0$, we have
\begin{equation*}
\begin{split}
&
\mathbf{E} [ f( X_{0,T}^{*}(x) ) ]
-
\mathbf{E} [ f( X_{0,n}^{n*}(x) ) ] \\
&=
\int_{0}^{+\infty}
f^{\prime} (y)
\left(
	\mathbf{P} ( X_{0,n}^{n} (y) > x )
	-
	\mathbf{P} ( X_{0,T} (y) > x )
\right)
\mathrm{d}y .
\end{split}
\end{equation*}
\end{Thm} 
This theorem is valid for
more generic
{\it stochastic flows on $[0, +\infty )$ in discrete-time}
(see Definition \ref{discrete-SF})
and will be proved in Section \ref{Conv>>W}
(see Theorem \ref{Weak_Err}).

For the strong error, we have the following:
\begin{Thm} 
\label{intro_main2} 
Assume the conditions (i)-(iv) in the subsection \ref{SFAB}.
If further $\sigma$ and $b$ are smooth
then $X_{0,T}^{*}$ is absolutely continuous
and we have for each $K > 0$,
\begin{equation*}
\begin{split}
&
\mathbf{E}
\left[
\sup_{0 \leq x \leq K}
\vert X_{0,T}^{*} (x) - X_{0,n}^{n*} (x) \vert
\right] \\
&\leq
\mathbf{E}
\left[
\left(
	1 + \sup_{0 \leq x \leq K} \vert ( X_{-T,0}^{*} )^{\prime} (x) \vert
\right)
\sup_{ 0 \leq x \leq m }
\vert
	X_{0,T} (x) - X_{0,n}^{n} (x)
\vert
\right] ,
\end{split}
\end{equation*}
where
$
m:= \min \{ X_{-T,0}^{*} (K), X_{-n,0}^{n*} (K) \}
$.
\end{Thm} 

This will be proved in Subsection \ref{CISE}.

\noindent
\underline{{\it Historical Background.}}

In \cite{Lev}, L\'{e}vy found the duality relation between
laws of the absorbing Brownian motion
$\{ \mathbf{P}_{x} \}_{x \geq 0}$
and the reflecting Brownian motion
$\{ \widehat{\mathbf{P}}_{y} \}_{y \geq 0}$
(both of their barriers are at zero):
\begin{equation*}
\begin{split}
\mathbf{P}_{x}
(
	X_{t} > y
)
=
\widehat{\mathbf{P}}_{y}
(
	X_{t} < x
),
\quad
\text{$x$, $y \geq 0$,}
\end{split}
\end{equation*}
where $X=(X_{t})_{t \geq 0}$ is the canonical process.
This is a typical form of the celebrated
{\it Siegmund's duality}
for absorbing and reflecting diffusions (\cite{Si}). 
This duality is usually formulated for two different diffusion measures
or their generators.

There is a lot of studies related to this kind of dualities.
An explicit appearance seems to date back at least to Karlin and McGregor's result (\cite[Section 6]{KMc}). 
The duality was observed between birth process and death process with interchanging their rates.
For other results, see \cite{CR}, \cite{DF}, \cite{DFPS} and \cite{Ja}.
We also refer the reader to \cite{JK} for a survey of this field.

We would like to note that notions of
Siegmund's duality also appear (with slight adaptive modifications) independently  in the study
of interacting particle systems.
For example,
in \cite[equation (3.18)]{Sp},
\cite{HL}
and
\cite[and see references therein]{Ha76},
the term ``associate" was used instead of ``dual".
It is worth referring to the idea of dual flow
appears in \cite{Wa07}.
In \cite{Wa07}, Warren explained
a duality between coalescing Brownian motions and
interlaced Brownian motions via dual flows,
in the study of Dyson's Brownian motion.

There are also sample paths approaches to Siegmund's duality. 
As was studied in \cite{ClSu} and \cite{Ha78},
it is described as follows.
$$
\mathbf{P}
( X_{t} > y \vert X_{0}=x )
=
\mathbf{P}
( \widehat{X}_{t} < x \vert \widehat{X}_{0}=y ) ,
\quad
\text{$x$, $y \geq 0$,}
$$
where $X$ and $\widehat{X}$ are Markov processes
defined on the same probability space.
Furthermore, as considered in \cite{AW} and \cite{WW},
this duality has a nice description
within the context of stochastic flows
associated with stochastic differential equations.
More precisely, for a given stochastic flow $\{ X_{s,t} \}_{s \leq t}$,
we can consider its
{\it dual stochastic flow} $\{ X_{s,t}^{*} \}_{s \leq t}$,
defined by
$
X_{s,t}^{*} := X_{-t, -s}^{-1}
$
(right-continuous inverse).
This might be a simple realization of
the so-called ``{\it graphical representation}"
in the field of dualities for Markov processes.

As explained in the beginning,
the dual flow $\{ X_{s,t}^{*} \}_{s \leq t}$
is the reflecting stochastic flow corresponding to a
Skorokhod-type stochastic differential equation 
under some assumptions on $\sigma$ and $b$.
Especially, what matters is that this method induces a transformation
for stochastic differential equations
describing the absorbing diffusion
and its dual reflecting diffusion.
Actually, the driving Wiener process $w$ of (\ref{Intr-Abs})
is transformed into
the time-reversal Wiener process $\widehat{w}$ which is the driving process for the corresponding
Skorokhod-type stochastic differential equation.
Hence we can deal with these two
stochastic differential equations
on the {\it same} probability space.
The dual flow method gives us
an algebraic-like approach to the pathwise uniqueness 
for Skorokhod-type stochastic differential equations.
This approach is completely different to existing literatures
such as \cite{LS}, \cite{Sa}, \cite{Ta} and so forth.
Indeed, in \cite{AW}, Akahori and Watanabe applied this method successfully
to prove the pathwise uniqueness 
for a large class of Skorokhod-type stochastic differential equations.
Although our stochastic flows are not diffeomorphisms,
we can partially apply Kunita's theory of stochastic flows of diffeomorphisms \cite{Ku}.
\\

\noindent
\underline{{\it The organization of this paper.}}

In Section 2, we introduce the notion of
stochastic flows on $[0, +\infty )$,
the dual stochastic flows and recall the duality
between them
in the context of stochastic differential equations 
(Theorem \ref{Principle}). 
We give a proof of this theorem in Appendix \ref{Appdx2}
in the case where
$\sigma \vert_{(0,+\infty)}$
and
$b \vert_{(0,+\infty)}$
are bounded.
The proof of Proposition \ref{Duality} is a bit complicated, 
and hence we put Proposition \ref{app-prop} in Appendix \ref{Appdx}
to prove that.
Furthermore, the proof of Proposition \ref{Duality}
uses also results from Subsection \ref{RC-inv_reg},
which can be read independently of Section 2.

In Section 3, we introduce the notion of
stochastic flows on $[0,+\infty )$ in discrete-time
and establish a discrete analogue of Theorem \ref{Principle}.

Section 4 and 5 are devoted to investigate the
convergence implications for
the weak error and strong error respectively.

\section{The duality relation}

Let
$w=(w(t))_{t \in \mathbb{R}}$ be a one-dimensional Wiener process.
For each $s \in \mathbb{R}$, we define
$w^{(s)} := ( w^{(s)} (t) )_{t\geq s}$
where
$w^{(s)} (t) := w(t) - w(s)$.

We denote by
$\mathcal{F}_{s,t}^{w}$
and 
$\mathcal{F}_{-\infty , t}^{w}$
the completion of
$\sigma ( w(u) - w(v) : s \leq v \leq u \leq t )$
and
$\vee_{-\infty < s \leq t} \mathcal{F}_{s,t}^{w}$,
respectively.

Along \cite{AW}, we review a duality relation between 
absorbing stochastic flows
and reflecting stochastic flows
associated with stochastic differential equations.
For this, we start with notions of
stochastic flows on $[0, +\infty )$.

\subsection{Stochastic flow on $[0, +\infty)$ and its dual stochastic flow}

Let $\mathcal{T}$ be the set of all non-decreasing and right-continuous
functions $\varphi : [0, +\infty ) \to [0, +\infty)$.
For a convention,
we define $\varphi (0-) := 0$ for $\varphi \in \mathcal{T}$.
The space $\mathcal{T}$ is a Polish space
under the metric
$
\rho ( \varphi , \psi )
:=
\sum_{n=1}^{\infty}
2^{-n} \min \{ \rho_{n} ( \varphi , \psi ) , 1 \}
$,
where
$$
\rho_{K} ( \varphi , \psi )
:=
\inf
\Big\{
	\varepsilon > 0:
	\begin{array}{c}
	\text{$
		\varphi ( x-\varepsilon ) - \varepsilon
		\leq
		\psi (x) \leq \varphi (x+\varepsilon ) + \varepsilon
	$} \\
	\text{for all $x \in [0,K]$}
	\end{array}
\Big\} ,
$$
for $K > 0$ and $\varphi$, $\psi \in \mathcal{T}$.
The space $\mathcal{T}$ is endowed with a semigroup structure
$
\mathcal{T} \times \mathcal{T} \ni ( \varphi , \psi )
\mapsto
\varphi \circ \psi \in \mathcal{T}
$.
We remark here that this operation and each evaluation map
$
\mathcal{T} \times [0, +\infty ) \ni ( \varphi , x )
\mapsto
\varphi (x) \in [0,+\infty )
$
are measurable but not continuous.

\begin{Def} 
\label{Def:Flow}
A family $\{ X_{s,t} \}_{s \leq t}$
of $\mathcal{T}$-valued random variables
is called a {\it stochastic flow} on $[0, +\infty )$
if it satisfies the following.
\begin{itemize}
\item[(i)]
(Flow property)
$
X_{s,u} = X_{t,u} \circ X_{s,t}
$
and
$X_{t,t} = \mathrm{id}_{[0,+\infty)}$
almost surely
for every $s \leq t \leq u$.

\item[(ii)]
(Independence)
For any sequence $t_{0} \leq t_{1} \leq \cdots \leq t_{n}$,
the $\mathcal{T}$-valued random variables
$\{X_{t_{k-1}, t_{k}}\}_{k=1}^{n}$ are independent.

\item[(iii)]
(Stationarity)
$X_{s,t} = X_{s+h, t+h}$ in law for each $s \leq t$ and $h >0$.

\item[(iv)]
(Stochastic continuity)
$X_{0,h} \to \mathrm{id}_{[0,+\infty )}$
in probability as $h \downarrow 0$.
\end{itemize}
\end{Def} 

For $\varphi\in\mathcal{T}$, 
we set
$
D( \varphi )
=
\{ x \geq 0 : \varphi (x-) \neq \varphi (x) \}
$
and
$$
R( \varphi )
:=
\Big\{
	\varphi (y):
	\begin{array}{c}
	\text{$y \geq 0$ and there exists $z > y$} \\
	\text{such that $\varphi (y) = \varphi (z)$}
	\end{array}
\Big\} .
$$
We give a sufficient condition to construct the dual flow of $\{ X_{s,t} \}_{s \leq t}$.

\begin{Prop} 
\label{Duality} 
Let
$\{ X_{s,t} \}_{s \leq t}$
be a stochastic flow on $[0,+\infty )$.
For $s \leq t$,
we define
$
X_{s,t}^{*} : [0, +\infty ) \to [0,+\infty )
$
by
$$
X_{s,t}^{*} (x)
:=
( X_{-t,-s} )^{-1} (x)
:=
\inf \{ y \in [0, +\infty ) : X_{-t,-s}(y) > x \} .
$$
If for each $s\leq t\leq u$,
$
\lim_{x \to +\infty} X_{s,t} (x) = +\infty
$
and
\begin{equation}
\label{RD-condi} 
R(X_{s,t}) \cap D(X_{t,u}) = \emptyset
\quad
\end{equation}
almost surely,
then $\{ X_{s,t}^{*} \}_{s \leq t}$ is a stochastic flow
on $[0, +\infty )$
and for each $s \leq t$,
$(X_{s,t}^{*})^{*} = X_{s,t}$
holds almost surely.
\end{Prop} 

\begin{proof}
From the definition of 
$\{ X_{s,t}^{*} \}_{s \leq t}$,
(ii) and (iii) of Definition \ref{Def:Flow} are clearly satisfied.
We show that $\{ X_{s,t}^{*} \}_{s \leq t}$ satisfies (i) and (iv) of Definition \ref{Def:Flow}.

We first show that (i).
$X_{t,t}^{*} = \mathrm{id}_{[0, +\infty)}$ is obvious from the definition of $X_{t,t}^{*}$.
Furthermore, from the right-continuity of $\{ X_{s,t} \}_{s \leq t}$,
it is enough to prove that
$$
X_{s,u}^{-1} (x) = X_{s,t}^{-1} \circ X_{t,u}^{-1} (x)
\quad
\text{almost surely}
$$
for $x \geq 0$ and $s \leq t \leq u$.
To prove this, we will use Proposition \ref{app-prop} in Appendix.
Proposition \ref{app-prop}-(ii) implies that
$
X_{s,u}^{-1} (x) \leq X_{s,t}^{-1} \circ X_{t,u}^{-1} (x)
$.
Now we suppose that
\begin{equation}
\label{contra-hyp} 
X_{s,u}^{-1} (x) < X_{s,t}^{-1} \circ X_{t,u}^{-1} (x).
\end{equation}
Then from Proposition \ref{app-prop}-(iii), we have 
\begin{equation}
\label{stay} 
X_{t,u}^{-1} ( x )
=
X_{s,t}( X_{s,u}^{-1} (x) )
=
X_{s,t} ( X_{s,t}^{-1} \circ X_{t,u}^{-1} (x) - ) .
\end{equation}
Now \eqref{contra-hyp} and \eqref{stay} imply that
$
X_{t,u}^{-1} (x) \in R( X_{s,t} )
$.
Therefore, from \eqref{RD-condi}, 
we see that
$
X_{t,u}^{-1} (x)
$
is a left-continuous point of $X_{t,u}$.
This contradicts to Proposition \ref{app-prop}-(iv).
Hence we can conclude that
$
X_{s,u}^{-1} (x) = X_{s,t}^{-1} \circ X_{t,u}^{-1} (x)
$.

Now we turn to prove (iv).
Let $K$ and $h$ be arbitrary positive numbers.
By Proposition \ref{inversion_reg} and \ref{LevKol},
we have
\begin{equation*}
\begin{split}
&
\rho_{K} ( \mathrm{id}_{[0, +\infty)}, X_{0,h}^{*} )
\leq
2
\sup_{0 \leq x \leq m}
\vert
	x - X_{-h,0}(x)
\vert ,
\end{split}
\end{equation*}
where $m := \min \{ K, X_{-h,0}(K)  \} \leq K$.
Moreover, by Proposition \ref{KolLev}, we have
\begin{equation*}
\begin{split}
&
\sup_{0 \leq x \leq m}
\vert
	x - X_{-h,0} (x)
\vert
\leq
\sup_{0 \leq x \leq K}
\vert
	x - X_{-h,0} (x)
\vert
\leq
2
\rho_{K} ( \mathrm{id}_{[0,+\infty)} , X_{-h,0} ) .
\end{split}
\end{equation*}
Therefore, we obtain that
\begin{equation*}
\begin{split}
\rho_{K} ( \mathrm{id}_{[0, +\infty)}, X_{0,h}^{*} )
\leq
4
\rho_{K} ( \mathrm{id}_{[0,+\infty)} , X_{-h,0} ) .
\end{split}
\end{equation*}
Now the stationarity for $\{ X_{s,t} \}_{s \leq t}$ yields that for any $\varepsilon > 0$,
\begin{equation*}
\begin{split}
&
\mathbf{P}
\big(
	\rho_{K} ( \mathrm{id}_{[0, +\infty)}, X_{0,h}^{*} ) > \varepsilon
\big) \\
&\leq
\mathbf{P}
\big(
	\rho_{K} ( \mathrm{id}_{[0, +\infty)}, X_{-h,0} ) > \varepsilon /4
\big) \\
&=
\mathbf{P}
\big(
	\rho_{K} ( \mathrm{id}_{[0, +\infty)}, X_{0,h} ) > \varepsilon /4
\big)
\end{split}
\end{equation*}
holds. 
Since the last term of the above converges to zero as $h$ tends to zero, 
we obtain the desired result.
\end{proof} 

\begin{Def} 
We call $\{ X_{s,t}^{*} \}_{s \leq t}$
the {\it dual stochastic flow} of $\{ X_{s,t} \}_{s \leq t}$.
\end{Def} 

\subsection{Stochastic flow with a reflecting barrier}
Let $\sigma$ and $b$ be Borel-measurable functions on $[0, +\infty )$. 
We consider Skorokhod-type stochastic differential equations
\begin{equation}
\label{SKR} 
\begin{split}
\left\{\begin{array}{l}
\mathrm{d}X_{t}
=
\sigma ( X_{t} ) \mathrm{d} w^{(s)}(t)
+
b ( X_{t} ) \mathrm{d}t
+
\mathrm{d} \phi (t),
\quad
t \geq s \\
X_{s} = x ,\quad x\geq 0.
\end{array}\right.
\end{split}
\end{equation}

We assume that the pathwise uniqueness holds for \eqref{SKR}.
Then for each $s \leq t$, we can define
$
X_{s,t} (x) := X_{t}
$,
where $(X_{u})_{u \geq s}$ is a unique strong solution to
\eqref{SKR} with $X_{s}=x$.
It is clear that each $X_{s,t}$ takes value in $\mathcal{T}$
and $\{ X_{s,t} \}_{s \leq t}$ forms a stochastic flow on $[0, +\infty )$.
We call this the
{\it $(\sigma , b, w)$-stochastic flow with the reflecting barrier at zero}.

\subsection{Stochastic flow with an absorbing barrier}
\label{SFAB} 
Let $\sigma$ and $b$ be Borel-measurable functions on $[0, +\infty )$. 
We assume the following:
\begin{itemize}
\item[(i)]
$\sigma (x) > 0$ for $x \in (0, +\infty )$.

\item[(ii)]
$\sigma \vert_{(0, +\infty)}$ and 
$b \vert_{(0, +\infty)}$
belong to
$C^{2}(0, +\infty )$
and their first derivatives are bounded on $[1, +\infty )$.

\item[(iii)]
\begin{equation}
\label{BRY} 
\int_{0+}^{1}
\left(
\frac{ 1 }{ \sigma (x)^{2} }
\exp
\left(
	- \int_{x}^{1}
	\frac{ 2 b(y) }{ \sigma (y)^{2} }
	\mathrm{d}y
\right)
+
\exp
\left(
	\int_{x}^{1}
	\frac{ 2 b(y) }{ \sigma (y)^{2} }
	\mathrm{d}y
\right)
\right)
\mathrm{d}x
< +\infty .
\end{equation}

\item[(iv)]
The condition (\ref{BRY}) still holds if we replace $b$ with $\widehat{b}$,
where
$$
\widehat{b}(x)
:=
\sigma (x) \sigma^{\prime} (x) - b(x).
$$
\end{itemize}

Under these assumptions,
we consider the stochastic differential equation:
\begin{equation}
\label{ORD} 
\begin{split}
\left\{\begin{array}{l}
\mathrm{d}X_{t}
=
\sigma ( X_{t} ) \mathrm{d}w^{(s)}(t)
+
b ( X_{t} ) \mathrm{d}t,
\quad
t \geq s, \\
X_{s} = x ,\quad x > 0.
\end{array}\right.
\end{split}
\end{equation}

Let $x>0$ and $s\in\mathbb{R}$.
Then above assumptions (i) and (ii) imply that
\eqref{ORD} admits the unique strong solution 
$X=(X_{t})_{t\geq s}$
with $X_{s}=x$, as long as $X$ moves in $(0,+\infty )$.
On the other hand, from (iii), the point zero is a regular boundary for the corresponding generator
$
L
=
( \sigma^{2} / 2 ) \frac{\mathrm{d}^{2}}{\mathrm{d} x^{2}}
+
b \frac{\mathrm{d}}{\mathrm{d} x}
$.
Hence there exists a finite random time $\tau^{s,x} > s$
such that
$
\lim_{t \to \tau^{s,x}} X_{t} = 0
$
almost surely. 
Now we define
$$
X_{s,t}^{-} (x)
:=
\left\{\begin{array}{ll}
X_{t} & \text{if $t \in [s, \tau^{s,x})$,} \\
0 & \text{if $t \geq \tau^{s,x}$}
\end{array}\right.
$$
for $x>0$ and $X_{s,t}^{-} (0) := X_{s,t}^{-} (0+)$.

\begin{Rm} 
Although the function
$
[s, +\infty ) \ni t \mapsto X_{s,t}^{-} (x)
$
is continuous for each fixed $s$ and $x$,
the function
$
( -\infty , t ] \ni s \mapsto X_{s,t}^{-} (x)
$
is highly discontinuous for each fixed $t$ and $x$.
The meaning of the latter process is unclear,
though there is an interpretation
(see \cite[Theorem 2]{WaS})
in the relation with Skorokhod's equation
when the situation is in the context of
coalescing stochastic flow generated by Tanaka's equation.
\end{Rm} 

\begin{Lem} 
\label{excit} 
With probability one, there exists $x_{0} > 0$
such that
$
X_{s,t}^{-} \vert_{(x_{0} , +\infty )}
$
is strictly increasing and 
$
X_{s,t}^{-} (x) = 0
$
for $0 \leq x < x_{0}$.
\end{Lem} 

\begin{proof} 
For simplicity, assume that $s=0$.
We show only that with probability one, there exists $x > 0$
such that
$
X_{0,t}^{-} (x) = 0
$.
Then the remaining is obvious.

Let $t$, $y>0$. 
Suppose that
$
\mathbf{P}
(
	\text{$X_{0,t}^{-} (x) > 0$ for all $x>0$}
)
>0
$
holds.
Then for any $T>t$, we have
\begin{equation*}
\begin{split}
&
\mathbf{P} ( X_{0,T}^{-} (y) > 0 )
=
\mathbf{E}
[
	\mathbf{P}
	(
		X_{T-t,T}^{-} \circ X_{0,T-t}^{-} (y) > 0
		\vert
		\mathcal{F}_{T-t, T}^{w}
	)
] \\
&=
\int_{ [0, +\infty ) }
\mathbf{P}
(
	X_{0,t}^{-} (x) > 0
)
\mathbf{P} ( X_{0,T-t}^{-} (y) \in \mathrm{d}x ) \\
&\geq
\mathbf{P}
(
	\text{$X_{0,t}^{-} (x) > 0$ for all $x>0$}
).
\end{split}
\end{equation*}
Therefore we see that
$$
\inf_{T > t}
\mathbf{P} ( X_{0,T}^{-} (y) > 0 )
> 0
$$
holds. 
However, since $\lim_{T \to \infty}X_{0,T}^{-} (y) = 0$,
by applying Fatou's lemma, we have
\begin{equation*}
\begin{split}
0
&=
\mathbf{P}
(
	\lim_{T \to \infty}X_{0,T}^{-} (y) > 0
) \\
&\geq
\limsup_{T \to \infty}
\mathbf{P}
(
	X_{0,T}^{-} (y) > 0
) \\
&\geq
\mathbf{P}
(
	\text{$X_{0,t}^{-} (x) > 0$ for all $x>0$}
)\\
&> 0.
\end{split}
\end{equation*}
This is a contradiction.
\end{proof} 

In particular we have
$
X_{t,u}^{-}( X_{s,t}^{-} (0) )
=
X_{s,u}^{-} (0)
$
for $s<t<u$.
Now we see that the family $\{ X_{s,t}^{-} \}_{s \leq t}$
forms a stochastic flow on $[0, +\infty )$.
We call this the
{\it $(\sigma , b, w)$-stochastic flow with the absorbing barrier at $0$}.

From Lemma \ref{excit}, we see that
$R ( X_{s,t}^{-} ) = \{ 0 \}$,
but
$
0\notin D( X_{s,t}^{-} )
$
for any $s \leq t$.
Therefore we have
$
R ( X_{s,t}^{-} ) \cap D ( X_{t,u}^{-} ) = \emptyset
$
for $s \leq t \leq u$.
Hence by Proposition \ref{Duality},
$\{ (X_{s,t}^{-})^{*} \}_{s \leq t}$
forms a dual stochastic flow.

\subsection{The duality relation}

Define
$\widehat{\sigma} := \sigma$,
$\widehat{b} := \sigma \sigma^{\prime} - b$
and
$
\widehat{w}
=
( \widehat{w}(t) )_{t \in \mathbb{R}}
$
as the time-reversal Wiener process given by
$
\widehat{w}(t)
:=
w(-t)
$.

Assume that (i)-(iv) in subsection \ref{SFAB} hold and the pathwise uniqueness holds for the Skorokhod-type stochastic differential equation:
\begin{equation}
\label{SKR^} 
\mathrm{d} X_{t}
=
\widehat{\sigma} ( X_{t} ) \mathrm{d} w^{(s)} (t)
+
\widehat{b} ( X_{t} ) \mathrm{d} t
+
\mathrm{d} \phi_{t},
\quad
t \geq s.
\end{equation}

\begin{Thm}[Akahori-Watanabe \cite{AW}, Warren-Watanabe \cite{WW}]
\label{Principle} 
Let
$\{ X_{s,t}^{-} \}_{s \leq t}$
be the 
$(\sigma , b, w)$-stochastic flow with the absorbing barrier at $0$.
Then the dual flow $\{ ( X_{s,t}^{-} )^{*} \}_{s \leq t}$
is the
$(\widehat{\sigma} , \widehat{b}, \widehat{w})$-stochastic flow
with the reflecting barrier at $0$ and vice versa.
\end{Thm} 

Theorem \ref{Principle} has already been stated in \cite{AW} and \cite{WW}, though it is not main result of them.
However the proof is given in neither (a rough sketch of the proof is given in \cite{AW}).
We give a proof of this theorem in Appendix \ref{Appdx2} when $\sigma$ and $b$ are bounded on $(0,+\infty)$.

\begin{Rm}
We note that the pathwise uniqueness assumption 
for \eqref{SKR^} is relaxed to the law-uniqueness.
In fact,
suppose that $(X, w^{\prime})$ is a solution to \eqref{SKR^} with $X_{0}=x$.
Then we can consider
$(\sigma , b, w)$-stochastic flow
$\{ X_{s,t}^{-} \}_{s\leq t}$
with the absorbing barrier at $0$,
where $w:=\widehat{w^{\prime}}$.
As we will see in Appendix \ref{Appdx2}, the process $\widehat{X}_{t} := (X_{0,t}^{-})^{*} (x)$ solves \eqref{SKR^}
and hence is a strong solution to (\ref{SKR^})
with $\widehat{X}_{0} = x$ and with respect to
$\widehat{w} = w^{\prime}$.

Now we have two solutions on the same probability space.
Since $\widehat{\sigma} >0$, $w^{\prime}$ is reconstructed from
$X$ and $\widehat{X}=(\widehat{X})_{t\geq 0}$
by the same procedure respectively.
Therefore the law-uniqueness for \eqref{SKR^}
implies that $(X,w^{\prime}) = (\widehat{X} , w^{\prime})$ in law.
Furthermore, since $\widehat{X}$ is a strong solution, 
$X=\widehat{X}$ holds almost surely. 
This proves the pathwise uniqueness for \eqref{SKR^}.
\end{Rm}

\section{A Discrete Analogue of the Duality Relation}

In this section, we establish a discrete analogues of
Theorem \ref{Principle} in the case
of the Euler-Maruyama approximation.

Let $w=(w(t))_{t \in \mathbb{R}}$ be a one-dimensional
Wiener process.
Fix $h > 0$, and let $t_{k} := kh$ for $k \in \mathbb{Z}$.
We write
$\Delta t_{k} := t_{k} - t_{k-1} \equiv h$
and
$\Delta w_{k} := w(t_{k}) - w(t_{k-1})$.

\begin{Def} 
\label{discrete-SF} 
We call a family of $\mathcal{T}$-valued random variables
$\{ X_{k,l} \}_{k \leq l}$
(where $k$ and $l$ move in $\mathbb{Z}$)
a {\it stochastic flow on $[0, +\infty )$ in discrete-time}
if it satisfies the following.
\begin{itemize}
\item[(i)]
$
X_{k,l} \circ X_{j,k} = X_{j,l}
$
and $X_{k,k} = \mathrm{id}_{[0, +\infty)}$ almost surely
for every integers $j \leq k \leq l$.

\vspace{2mm}
\item[(ii)]
For any sequence of integers
$k_{0} \leq k_{1} \leq \cdots \leq k_{m}$, 
$\mathcal{T}$-valued random variables
$(X_{k_{i-1},k_{i}})_{i=1}^{m}$
are independent.

\vspace{2mm}
\item[(iii)]
$X_{k,l} = X_{k+a,l+a}$ in law
for each $k \leq l$ and $a\in\mathbb{N}$.

\end{itemize}

\end{Def} 
Similarly to the previous section, for a given stochastic flow 
$\{ X_{k,l} \}_{k\leq l}$ on $[0,+\infty)$ in discrete-time, 
we define its dual $\{ X_{k,l}^{*} \}_{k\leq l}$ by 
$X_{k,l}^{*}(x):=\inf\{ y\geq 0: X_{-l,-k}^{-1}(y)>x \}$, $x\geq 0$.

We consider the Euler-Maruyama approximation for
the absorbing barrier diffusion associated with (\ref{ORD})
as follows.

For $k,l\in\mathbb{Z}$ with $k \leq l$ and $x \in (0, +\infty )$,
we define
$$X_{k,k}^{-} (x) := x$$
and for $k < l$,
\begin{equation*}
\begin{split}
X_{k,l}^{-} (x)
&
:=
1_{ (0,+\infty) } (X_{k,l-1}^{-} (x)) \\
&\hspace{5mm}\times
\max
\big\{
	0,
	X_{k,l-1}^{-} (x)
	+
	\sigma ( X_{k,l-1}^{-} (x) ) \Delta w_{l}
	+
	b ( X_{k,l-1}^{-} (x) ) \Delta t_{l}
\big\} .
\end{split}
\end{equation*}
For $x=0$, we define $X_{k,l}^{-} (0) := X_{k,l}^{-} (0+)$.
By the construction, for each $k \in \mathbb{Z}$,
$X_{k,k+1}^{-} \vert_{(0, +\infty )}$
is continuous.
Furthermore, if
\begin{equation}
\label{Tech}
\text{{\it $X_{k,l}^{-} : [0, +\infty ) \to [0,+\infty )$
is non-decreasing almost surely for each $k \leq l$,}}
\end{equation}
then 
$\{ X_{k,l}^{-} \}_{k \leq l}$
forms a stochastic flow on $[0, +\infty)$
in discrete-time in the sense of Definition \ref{discrete-SF}.
\begin{Def} 
\label{EM-AbsFl} 
We call $\{ X_{k,l}^{-} \}_{k\leq l}$
the {\it Euler-Maruyama approximation} of the
$(\sigma ,b, w)$-stochastic flow with the absorbing barrier
at $0$
if the condition (\ref{Tech}) holds.
\end{Def} 

We give a simple sufficient condition for (\ref{Tech}).

\begin{Prop} 
\label{suff_eg} 
Assume that $\sigma$ and $b$ satisfy the following three conditions.
\begin{itemize}
\item[(1)]
$\sigma \vert_{(0, +\infty)}$
and
$b \vert_{(0, +\infty)}$
are continuously differentiable.

\vspace{2mm}
\item[(2)]
$\sigma (x) > 0$
and
$\sigma^{\prime} (x) \geq 0$
for $x>0$.

\vspace{2mm}
\item[(3)]
For each $x>0$, it holds that
$$
x ( \log \vert \sigma \vert )^{\prime} (x)
+
b(x)
\big\{
	( \log \vert \sigma \vert )^{\prime} (x)
	-
	( \log \vert b \vert )^{\prime} (x)
\big\}
\Delta t_{l}
\leq 1 .
$$
\end{itemize}
Then $X_{k,l}^{-}$ is non-decreasing almost surely.
\end{Prop} 

\begin{proof}[Proof of Proposition \ref{suff_eg}.] 
It is enough to prove the non-decreaseness of
$X_{l-1,l}^{-} \vert_{(0,+\infty)}$
for each $l\in\mathbb{Z}$.
Put
$$
h(x)
:=
x + \sigma (x) \Delta w_{l} + b(x) \Delta t_{l},
\quad
x > 0.
$$
Since
$
X_{l-1,l}^{-} (x) = \max \{ 0, h(x) \}
$
is absolutely continuous,
it suffices to show that
$
( X_{l-1,l}^{-} )^{\prime} (x) \geq 0
$
for almost all $x > 0$.

Since $h(x)$ has the density for all $x>0$,
we have
$\mathbf{P}(\{h(x)=0\})=0$ for all $x>0$.
Therefore, from Fubini's theorem, we have
$$
\mathbf{E}
[
	\mathrm{Leb} ( \{ h = 0 \} )
]
=
\int_{0}^{+\infty} \mathbf{P} ( h(x) = 0 ) \mathrm{d}x
= 0.
$$
Hence $\{ x> 0 : h(x) = 0 \}$
is of Lebesgue measure zero almost surely.

Now it only needs to show the non-negativity of
$( X_{l-1,l}^{-} )^{\prime}$
on the two open sets
$\{ h < 0 \}$
and
$\{ h > 0 \}$.

If $x \in \{ h < 0 \}$, then
we clearly have
$
( X_{l-1,l}^{-} )^{\prime} (x) = 0
$.

On the other hand, if $x \in \{ h > 0 \}$, then the condition (2) yields that
$$
\Delta w_{l}
>
\frac{ -1 }{ \sigma (x) }
\left(
	x + b(x) \Delta t_{l}
\right).
$$
By using the condition (3), we obtain that
\begin{equation*}
\begin{split}
&
( X_{l-1,l}^{-} )^{\prime} (x)
=
1 +
\sigma^{\prime} (x)
\Delta w_{l}
+ b^{\prime} (x) \Delta t_{l} \\
&\geq
1 +
\sigma^{\prime} (x)
\frac{ -1 }{ \sigma (x) }
\left(
	x + b(x) \Delta t_{l}
\right)
+ b^{\prime} (x) \Delta t_{l} \\
&=
1-
\left(
x ( \log \vert \sigma \vert )^{\prime} (x)
+
b(x)
\left(
	( \log \vert \sigma \vert )^{\prime} (x)
	-
	( \log \vert b \vert )^{\prime} (x)
\right) \Delta t_{l}
\right) \\
&\geq 0.
\end{split}
\end{equation*}
This completes the proof.
\end{proof} 

\begin{Eg} 
The following pairs fulfill
assumptions in Proposition \ref{suff_eg}.
\begin{itemize}
\item[$\bullet$]
$\sigma (x) = \alpha x + \beta$
and
$b (x) = \gamma x + \delta$
with
$$
\alpha \geq 0,
\quad
\beta \geq 0
\quad
\text{and}
\quad
\det
\left(\begin{array}{cc}
	\alpha & \beta \\
	\gamma & \delta
\end{array}\right)
\leq
\beta .
$$

\item[$\bullet$]
In particular,
$\sigma (x) = \alpha x + \beta$
and
$b (x) = \gamma x + \delta$
with
$$
\alpha \geq 0,
\quad
\beta \geq 0,
\quad
\gamma \geq 0
\quad
\text{and}
\quad
\delta \leq 0.
$$

\item[$\bullet$]
$\sigma (x) = \alpha x + \beta$
($\alpha \geq 0$, $\beta \geq 0$)
and
$b(x) = - 1/x$.

\vspace{2mm}
\item[$\bullet$]
The pair
$
\sigma (x) = \sqrt{2ax}
$
($a>0$)
and
$
b(x) = cx + d
$
($c \in \mathbb{R}$, $d \leq 0$)
also satisfies the conditions in Proposition \ref{suff_eg}
if $n$ is large enough.

\end{itemize}
\end{Eg} 

We work with the above technical assumption
(\ref{Tech})
in the sequel.
We define the stochastic flow
$\{ \widehat{X}_{k,l} \}_{k \leq l}$
on $[0, +\infty)$
in discrete-time by
$$
\widehat{X}_{k,l} := ( X_{-l,-k}^{-} )^{-1}.
$$

\begin{Prop} 
\label{max-eq} 
Let $k \in \mathbb{Z}$ and $x \geq 0$ and 
assume that \eqref{Tech} holds.
If
$\widehat{X}_{k,k+1}(x) > 0$ then we have
$X_{-k-1, -k}^{-} ( \widehat{X}_{k,k+1}(x) ) = x$
and
\begin{equation*}
\begin{split}
&
x
=
\max
\big\{
	0,
	\widehat{X}_{k,k+1}(x)
	- \sigma ( \widehat{X}_{k,k+1}(x) ) \Delta \widehat{w}_{k+1}
	+ b( \widehat{X}_{k,k+1}(x) ) \Delta t_{k+1}
\big\} .
\end{split}
\end{equation*}
\end{Prop} 
\begin{proof} 
Since
$\Delta w_{-k} = - \Delta \widehat{w}_{k+1}$
and
$\Delta t_{-k} = \Delta t_{k+1}$,
for any $y > 0$, we have
\begin{equation*}
\begin{split}
&
\max
\{
	0,
	y
	- \sigma (y) \Delta \widehat{w}_{k+1}
	+ b(y) \Delta t_{k+1}
\} \\
&=
\max\{
	0,
	y + \sigma (y) \Delta w_{-k} + b(y) \Delta t_{-k}
\}
=
X_{-k-1, -k}^{-} (y).
\end{split}
\end{equation*}
By substituting $y=\widehat{X}_{k,k+1}(x) > 0$ for this equation,
we obtain that
\begin{equation*}
\begin{split}
&
X_{-k-1, -k}^{-} ( \widehat{X}_{k,k+1}(x) ) \\
&=
\max
\big\{
	0,
	\widehat{X}_{k,k+1}(x)
	- \sigma ( \widehat{X}_{k,k+1}(x) ) \Delta \widehat{w}_{k+1}
	+ b( \widehat{X}_{k,k+1}(x) ) \Delta t_{k+1}
\big\} .
\end{split}
\end{equation*}
We turn to prove
$X_{-k-1, -k}^{-} ( \widehat{X}_{k,k+1}(x) ) = x$. 
Since
$X_{-k-1, -k}^{-} ( \widehat{X}_{k,k+1}(x) ) \geq x$,
from the definition, it suffices to prove 
$X_{-k-1, -k}^{-} ( \widehat{X}_{k,k+1}(x) ) \leq x$.

Take an arbitrary $\varepsilon \in (0, \widehat{X}_{k, k+1}(x))$.
Then there exists $y \in [0, +\infty )$ such that
$X_{-k-1,-k}^{-} (y) > x$
and
$$\widehat{X}_{k,k+1} (x) + \varepsilon > y  > \widehat{X}_{k,k+1} (x).$$
Then noting that $y-\varepsilon > 0$ but
$y-\varepsilon \notin \{ z \geq 0 : X_{-k-1,-k}^{-} (z) > x \}$,
we see that
\begin{equation*}
\begin{split}
X_{-k-1,-k}^{-}
\big(
	\widehat{X}_{k, k+1}(x) - \varepsilon
\big)
\leq
X_{-k-1,-k}^{-}
(
	y - \varepsilon
)
\leq
x.
\end{split}
\end{equation*}
Since $X_{-k-1,-k}^{-}$ is continuous on $(0,+\infty )$
and $\varepsilon > 0$ is arbitrary,
we find that
$
X_{-k-1,-k}^{-}
\big(
	\widehat{X}_{k, k+1}(x)
\big)
\leq
x
$.
\end{proof} 

\begin{Cor} 
\label{EM-ref}
Suppose that
$\sigma$ and $b$ are constants.
Then for any $k \in \mathbb{Z}$ and $x \geq 0$,
we have
\begin{equation*}
\begin{split}
\widehat{X}_{k,k+1}(x)
=
\max
\{
	0,
	x + \sigma \Delta \widehat{w}_{k+1} - b \Delta t_{k+1}
\} .
\end{split}
\end{equation*}
\end{Cor} 
\begin{proof} 
Suppose that $\widehat{X}_{k,k+1}(x) > 0$.
Then by Proposition \ref{max-eq}, we have, 
\begin{equation*}
\begin{split}
&
x
=
\max
\big\{
	0,
	\widehat{X}_{k,k+1}(x)
	- \sigma \Delta \widehat{w}_{k+1}
	+ b \Delta t_{k+1}
\big\} .
\end{split}
\end{equation*}
Therefore, if $x > 0$, we have
$
\widehat{X}_{k,k+1}(x)
=
x + \sigma \Delta \widehat{w}_{k+1} - b \Delta t_{k+1}
$.

If $x=0$ then, by the definition,
\begin{equation*}
\begin{split}
\widehat{X}_{k,k+1}(0)
&=
\inf \{ y: y + \sigma \Delta w_{-k} + b \Delta t_{-k} > 0 \} \\
&=
\inf
\Big(
[0, +\infty )
\cap
\big(
	\sigma \Delta \widehat{w}_{k+1} + b \Delta t_{k+1},
	+\infty
\big)
\Big) \\
&=
\max
\{
	0,
	\sigma \Delta \widehat{w}_{k+1} + b \Delta t_{k+1}
\} .
\end{split}
\end{equation*}
Since $\widehat{X}_{k,k+1}(0) > 0$,
we have
$
\widehat{X}_{k,k+1}(0)
=
\sigma \Delta \widehat{w}_{k+1} + b \Delta t_{k+1}
$.

\end{proof} 

\section{Convergence Implication for Weak Error}
\label{Conv>>W} 

In this section, we prove Theorem \ref{Weak_Err}.
To prove this theorem, we prepare several propositions.

\begin{Prop} 
\label{seesaw} 
Let $\varphi \in \mathcal{T}$.
Then for $x,y \geq 0$, we have
\begin{itemize}
\item[{\rm (i)}]
$
\varphi^{-1} (y) \leq x
\Longrightarrow
\varphi (x) \geq y
$,

\item[{\rm (ii)}]
$
\varphi (x) > y
\Longrightarrow
\varphi^{-1} (y) \leq x
$.
\end{itemize}
\end{Prop} 

The proof is easy and hence omitted.
The next proposition says that the dual flow method realizes 
Siegmund's duality.

\begin{Prop} 
\label{S-duality} 
Let $\{ X_{s,t} \}_{s \leq t}$ be a stochastic flow on $[0, +\infty )$.
Assume that
$
\mathbf{P} ( X_{s,t} (x) = y ) = 0
$
holds for $x,y > 0$.
Then for each
$
x_{i},
y_{i} > 0
$,
$i=1,\cdots ,n$,
we have 
\begin{equation*}
\begin{split}
&\mathbf{P}
(
	X_{s,t}^{*} (y_{1}) \leq x_{1} ;
	\cdots ;
	X_{s,t}^{*} (y_{n}) \leq x_{n}
)\\
&=
\mathbf{P}
(
	X_{s,t} (x_{1}) > y_{1} ;
	\cdots ;
	X_{s,t} (x_{n}) > y_{n}
) .
\end{split}
\end{equation*}
The same is true for
stochastic flows on $[0, +\infty )$
in discrete-time.
\end{Prop} 
\begin{proof} 
By Proposition \ref{seesaw},
we see that
\begin{equation*}
\begin{split}
&
\{
	X_{s,t}^{*} (y_{1}) \leq x_{1} ;
	\cdots ;
	X_{s,t}^{*} (y_{n}) \leq x_{n}
\} \\
&\subset
\{
	X_{-t,-s} (x_{1}) \geq y_{1} ;
	\cdots ;
	X_{-t,-s} (x_{n}) \geq y_{n}
\} .
\end{split}
\end{equation*}
Therefore, by using the stationarity and assumption, we have
\begin{equation*}
\begin{split}
&
\mathbf{P}
(
	X_{s,t}^{*} (y_{1}) \leq x_{1} ;
	\cdots ;
	X_{s,t}^{*} (y_{n}) \leq x_{n}
) \\
&\leq
\mathbf{P}
(
	X_{s,t} (x_{1}) \geq y_{1} ;
	\cdots ;
	X_{s,t} (x_{n}) \geq y_{n}
) \\
&=
\mathbf{P}
(
	X_{s,t} (x_{1}) > y_{1} ;
	\cdots ;
	X_{s,t} (x_{n}) > y_{n}
) .
\end{split}
\end{equation*}
Similarly we have
\begin{equation*}
\begin{split}
&
\mathbf{P}
(
	X_{s,t} (x_{1}) > y_{1} ;
	\cdots ;
	X_{s,t} (x_{n}) > y_{n}
) \\
&\leq
\mathbf{P}
(
	X_{s,t}^{*} (y_{1}) \leq x_{1} ;
	\cdots ;
	X_{s,t}^{*} (y_{n}) \leq x_{n}
) .
\end{split}
\end{equation*}
From Proposition \ref{seesaw}, we get the desired result.
\end{proof} 

\begin{Prop} 
\label{Formula-Weak} 
Let $X$ be a $\mathcal{T}$-valued random variable
such that $\lim_{x \to +\infty} X(x) = +\infty$ almost surely.
Let $x > 0$ and
$f : [0, +\infty ) \to \mathbb{R}$
be a differentiable functions
with $f(0)=0$ and with compact support.
If $\mathbf{P}( X (y) = x ) = 0$ for each $y > 0$,
then we have
\begin{equation*}
\begin{split}
&
\mathbf{E} [
f ( X^{-1} ( x ) )
]
=
-
\int_{0}^{+\infty}
f^{\prime} ( y )
\mathbf{P}
(
	X ( y ) > x
)
\mathrm{d}y .
\end{split}
\end{equation*}
\end{Prop} 
\begin{proof} 
Since $X^{-1}(x)$ is a non-negative random variable,
we have
\begin{equation*}
\begin{split}
&
\mathbf{E} [
f ( X^{-1} ( x ) )
]
=
\int_{[0, +\infty )}
f ( y )
\mathbf{P}
(
	X^{-1} ( x ) \in \mathrm{d} y
).
\end{split}
\end{equation*}
Since $f(0)=0$ and the support of $f$ is compact, 
an integration by parts formula
for the Lebesgue-Stieltjes integral
yields that
\begin{equation*}
\begin{split}
&
\int_{[0, +\infty )}
f ( y )
\mathbf{P}
(
	X^{-1} ( x ) \in \mathrm{d} y
) \\
&=
\lim_{b \to +\infty}
\mathbf{P}
(
	X^{-1} ( x ) \leq b
)
f(b)
-
\mathbf{P}
(
	X^{-1} ( x ) \leq 0
)
f(0) \\
&\hspace{20mm} -
\int_{0}^{+\infty}
f^{\prime} ( y )
\mathbf{P}
(
	X^{-1} ( x ) \leq y
)
\mathrm{d}y \\
&=
-
\int_{0}^{+\infty}
f^{\prime} ( y )
\mathbf{P}
(
	X^{-1} ( x ) \leq y
)
\mathrm{d}y.
\end{split}
\end{equation*}
Now, from Proposition \ref{S-duality}, we obtain the result.
\end{proof} 

From Proposition \ref{Formula-Weak} and the stationarity of stochastic flows,
we obtain our first main result.

\begin{Thm} 
\label{Weak_Err} 
Let $x > 0$, $T>0$ and $n \in \mathbb{N}$.
Assume that a stochastic flow on $[0, +\infty )$
$\{ X_{s,t} \}_{s \leq t}$ 
and a stochastic flow on $[0, +\infty )$ in discrete-time $\{ Y_{k,l} \}_{k \leq l}$ satisfy 
$$
\lim_{y \to +\infty} X_{0,T} (y)
=
\lim_{y \to +\infty} Y_{0,n} (y)
= + \infty
$$
and
$$
\mathbf{P} ( X_{0,T} (y) = x )
=
\mathbf{P} ( Y_{0,n} (y) = x )
= 0
\quad
\text{for any $y > 0$.}
$$
Then for each differentiable function
$f : [0, +\infty ) \to \mathbb{R}$
with compact support and satisfies $f(0)=0$,
we have
\begin{equation*}
\begin{split}
&
\mathbf{E} [ f( X_{0,T}^{*} (x) ) ]
-
\mathbf{E} [ f( Y_{0,n}^{*} (x) ) ] \\
&=
\int_{0}^{+\infty}
f^{\prime} ( y )
\left(
\mathbf{P} ( Y_{0,n} ( y ) > x )
-
\mathbf{P} ( X_{0,T} ( y ) > x )
\right)
\mathrm{d}y .
\end{split}
\end{equation*}
\end{Thm} 

We now focus on the Euler-Maruyama approximation 
of $\{ X_{s,t} \}_{s\leq t}$ (see Definition \ref{EM-AbsFl}).

Assume the conditions (i)-(iv) in the subsection \ref{SFAB}.
Let
$\{ X_{s,t} \}_{s\leq t}$
be a
$(\sigma , b, w)$-stochastic flow
with the absorbing barrier at $0$. 
Fix $T>0$, $n \in \mathbb{N}$ and put $h := T/n$.
We denote by
$\{ X_{k,l}^{n} \}_{s\leq t}$
the associated Euler-Maruyama approximation
of $\{ X_{s,t} \}_{s\leq t}$.

By Theorem \ref{Principle},
the dual flow
$\{ X_{s,t}^{*} \}_{s \leq t}$
of
$\{ X_{s,t} \}_{s\leq t}$
is a
$(\widehat{\sigma} , \widehat{b}, \widehat{w})$-stochastic flow
with the reflecting barrier at $0$.
Then the dual flow
$\{ X_{k,l}^{n*} \}_{s\leq t}$
of
$\{ X_{k,l}^{n} \}_{s\leq t}$
might be naturally an approximation of
$\{ X_{s,t}^{*} \}_{s \leq t}$.
By Theorem \ref{Weak_Err},
we have for $x > 0$,
\begin{equation*}
\begin{split}
&
\mathbf{E} [ f( X_{0,T}^{*} (x) ) ]
-
\mathbf{E} [ f( X_{0,n}^{n*} (x) ) ] \\
&=
\int_{0}^{+\infty}
f^{\prime} ( y )
\left(
\mathbf{P} ( X_{0,n}^{n} ( y ) > x )
-
\mathbf{P} ( X_{0,T} ( y ) > x )
\right)
\mathrm{d}y .
\end{split}
\end{equation*}
Hence, to know the rate of convergence for
$
\vert
\mathbf{E} [ f(X_{T}^{*}(x)) ]
-
\mathbf{E} [ f(X_{T}^{n*}(x)) ]
\vert
$,
where
$
X_{T}^{*}(x) := X_{0,T}^{*}(x)
$
and
$
X_{T}^{n*}(x) := X_{0,n}^{n*} (x)
$,
it suffices to estimate
$
\vert
\mathbf{P} ( X_{0,T}(y) > x )
-
\mathbf{P} ( X_{0,n}^{n}(y) > x )
\vert
$
{\it with clarifying the dependence on the initial point $y>0$}.

Gobet~\cite{Gobet00} investigated the weak order for
the Euler-Maruyama approximations
of killed diffusions.
Suppose that there exist
$\widetilde{\sigma}$ and $\widetilde{b}: \mathbb{R} \to \mathbb{R}$
such that

(1)
$\widetilde{\sigma}\vert_{(0, +\infty )} = \sigma$
and
$\widetilde{b}\vert_{(0, +\infty )} = b$,

(2)
$\widetilde{\sigma}$ and $\widetilde{b}$ are smooth,

(3)
$\inf_{x \in \mathbb{R}} \widetilde{\sigma} (x)^{2} > 0$

and

(4)
all derivatives of $\widetilde{\sigma}$ and $\widetilde{b}$ are bounded.

\noindent Under these conditions,
his result
(\cite[Theorem 2.3]{Gobet00})
tells us that for each $x > 0$ and $y>0$,
there exists a constant
$K(T,y)>0$ such that
\begin{equation*}
\vert
\mathbf{P} ( X_{T}(y) > x )
-
\mathbf{P} ( X_{T}^{n}(y) > x )
\vert
\leq
K(T,y) n^{-1/2} ,
\end{equation*}
where
$
X_{T} (y) := X_{0,T} (y)
$
and
$
X_{T}^{n} (y) := X_{0,n}^{n}(x)
$.
Furthermore, with a careful reading of the proof,
we see that $K(T,y)$ can be rearranged to be
continuous in $y > 0$.
Hence we obtain the following estimate.

\begin{Cor} 
Assume the conditions (1)-(4) above.
Suppose that 
$f : [0, +\infty ) \to \mathbb{R}$ is
a differentiable function with compact support and satisfies $f(0)=0$.
Then for each $x>0$, we have 
\begin{equation*}
\begin{split}
\vert
\mathbf{E} [ f( X_{T}^{*}(x)) ]
-
\mathbf{E} [ f( X_{T}^{n*}(x)) ]
\vert
\le
\frac{C(T) n^{-1/2}}{ \min\{ 1,(x/2)^{2} \} },
\end{split}
\end{equation*}
where
$
C(T):=\int_{0}^{+\infty} \vert f'(y) \vert K(T,y)dy.
$
\end{Cor} 

\section{Convergence Implication for Strong Error}

In this section, we first estimate the quantity 
$\sup_{0 \leq x \leq K}\vert \varphi^{-1} (x) - \psi^{-1} (x) \vert$
for each $\varphi,\psi$ in $\mathcal{T}$ and $K>0$ 
(see Proposition \ref{inversion_reg}). 
By using this estimate, we prove Theorem \ref{str-err}.

\subsection{Continuity of taking the right-continuous inverse}
\label{RC-inv_reg} 

Let $\varphi$ and $\psi$ be elements of $\mathcal{T}$.
For $K > 0$,
we define
$$
\rho_{K} ( \varphi , \psi )
:=
\inf
\Big\{
	\varepsilon > 0 :
	\text{
	$
	\begin{array}{l}
	\varphi ( x-\varepsilon ) - \varepsilon
	<
	\psi (x), \\
	\psi ( x-\varepsilon ) - \varepsilon
	<
	\varphi (x)
	\end{array}
	$
	for all $x \in [0,K]$}
\Big\} ,
$$
where the values of $\varphi$ and $\psi$ on $(-\infty , 0)$
are understood to be zero.

In this subsection, we assume that
$$
\lim_{x \to +\infty} \varphi (x)
=
\lim_{x \to +\infty} \psi (x)
=
+\infty .
$$
Under these conditions, we see that
the right-continuous inverses
$\varphi^{-1}$ and $\psi^{-1} $ belong to $\mathcal{T}$.


\begin{Prop} 
\label{inversion_reg} 
Let $\varphi$ and $\psi$ be in $\mathcal{T}$ and $K>0$. 
If $\varphi^{-1}$ is absolutely continuous then we have
\begin{equation*}
\begin{split}
&
\sup_{0 \leq x \leq K}
\vert \varphi^{-1} (x) - \psi^{-1} (x) \vert \\
&\leq
\left(
	1 + \sup_{0 \leq x \leq K} \vert ( \varphi^{-1} )^{\prime} (x) \vert
\right)
\sup_{ 0 \leq x \leq m }
\vert
	\varphi (x) - \psi (x)
\vert ,
\end{split}
\end{equation*}
where
$
m:= \min \{ \varphi^{-1} (K), \psi^{-1} (K) \}
$.
\end{Prop} 

To prove this proposition, we need several estimates.

\begin{Prop} 
\label{LevKol} 
For any $\varphi$, $\psi$ in $\mathcal{T}$ and $K>0$, we have
$$
\rho_{K} ( \varphi , \psi )
\leq
\sup_{0 \leq x \leq K}
\vert
	\varphi (x) - \psi (x)
\vert.
$$
\end{Prop} 
\begin{proof} 
Let
$
l:=
\sup_{0 \leq x \leq K}
\vert
	\varphi (x) - \psi (x)
\vert
$
and $\delta$ be an arbitrary positive number.
Then it is easy to see that
$
- ( l + \delta )
<
\psi (x) - \varphi (x)
<
l + \delta
$
for all $x \leq K$.
Therefore we have
$$
\left\{
\begin{array}{l}
\varphi (x) - (l+\delta )
<
\psi (x) \\
\psi (x) - (l+\delta )
<
\varphi (x)
\end{array}
\right.
\quad
\text{for all $x \leq K$.}
$$
Furthermore, since $\varphi$ and $\psi$ are non-decreasing,
we also have
$$
\left\{
\begin{array}{l}
\varphi ( x - (l+\delta) ) - (l+\delta )
<
\psi (x) \\
\psi ( x - (l+\delta ) ) - (l+\delta )
<
\varphi (x)
\end{array}
\right.
\quad
\text{for all $x \leq K$.}
$$
This implies that
$
\rho_{K} ( \varphi , \psi )
\leq
l + \delta
=
\sup_{0 \leq x \leq K}
\vert
	\varphi (x) - \psi (x)
\vert
+
\delta
$.
Since $\delta > 0$ is arbitrary,
we obtain the result.
\end{proof} 

\begin{Prop} 
\label{KolLev} 
Assume that $\varphi$ and $\psi$ belong to $\mathcal{T}$. 
If either $\varphi$ or $\psi$ is absolutely continuous
(say $\varphi$), then for any $K>0$, we have
$$
\sup_{ 0 \leq x \leq K}
\vert
	\varphi (x) - \psi (x)
\vert
\leq
\big\{
	1 +
	\sup_{
		x \leq K
	}
	\vert \varphi^{\prime} (x) \vert
\big\}
\rho_{K} ( \varphi , \psi ) .
$$
\end{Prop} 
\begin{proof} 
By the definition of $\rho_{K}$,
we can find some decreasing sequence
$\varepsilon_{n}$ which converges to
$\rho_{K} ( \varphi , \psi )$
and satisfies that
$$
\left\{\begin{array}{ll}
\varphi ( x-\varepsilon_{n} ) - \psi (x) < \varepsilon_{n},
&
\text{for any $x \leq K$ and $n\in\mathbb{N}$,}
\\
\psi (y) - \varphi ( y+\varepsilon_{n} ) < \varepsilon_{n}
&
\text{for any $y \leq K + \varepsilon_{n}$ and $n\in\mathbb{N}$.}
\end{array}\right.
$$
Since $\varphi$ is continuous,
by letting $n \to \infty$,
we have
\begin{equation}
\label{eq:4} 
\left\{\begin{array}{ll}
\varphi ( x - \rho_{K} ( \varphi , \psi ) )
-
\psi (x)
\leq
\rho_{K} ( \varphi , \psi ), \\
\psi (x)
-
\varphi ( x + \rho_{K} ( \varphi , \psi ) )
\leq
\rho_{K} ( \varphi , \psi )
\end{array}\right.
\quad
\text{for any $x \leq K$.}
\end{equation}
By using the fundamental theorem of calculus
for Lebesgue integral,
we find that
\begin{equation*}
\varphi ( x - \rho_{K} ( \varphi , \psi ) )
=
\varphi ( x )
+
\int_{ x }^{ x-\rho_{K} ( \varphi , \psi ) }
\varphi^{\prime} (y) \mathrm{d}y
\end{equation*}
and
\begin{equation*}
\varphi ( x + \rho_{K} ( \varphi , \psi ) )
=
\varphi ( x )
+
\int_{ x }^{ x+\rho_{K} ( \varphi , \psi ) }
\varphi^{\prime} (y) \mathrm{d}y.
\end{equation*}
Now (\ref{eq:4}) completes the proof.
\end{proof} 

\begin{Prop} 
\label{RegInv} 
Let $\varphi$ and $\psi$ belong to $\mathcal{T}$ and $K>0$. 
Define
$L := \max \{ \varphi^{-1} (K), \psi^{-1} (K) \}$
and
$f_{K} := \min \{ f , K \}$
for $f \in \mathcal{T}$.
Then we have
$$
\rho_{K} ( \varphi^{-1} , \psi^{-1} )
\leq
\rho_{L} ( \varphi_{K} , \psi_{K} ).
$$
\end{Prop} 
\begin{proof} 
We first note that,
for any non-decreasing and right-continuous function
$A: [0, +\infty ) \to \mathbb{R}$,
it holds that
\begin{align}
\label{Ineq:01} 
A(A^{-1}(x)) \geq x
\end{align}
for each $x \in [0, +\infty )$.

Let $\varphi$, $\psi\in\mathcal{T}$,
$
L:=\max \{ \varphi^{-1}(K) , \psi^{-1} (K) \}
$
and
$$
\mathcal{L}
:=
\Big\{
	\varepsilon > 0 :
	\text{$
	\begin{array}{c}
	\varphi_{K} ( x-\varepsilon ) - \varepsilon
	<
	\psi_{K} (x) \\
	\psi_{K} ( x-\varepsilon ) - \varepsilon
	<
	\varphi_{K} (x)
	\end{array}
	$
	for all $x \leq L$
	}
\Big\}.
$$
Then for any $\varepsilon \in \mathcal{L}$ and $x \leq L$, we have
\begin{align}
&\label{eq:6} 
\varphi_{K} ( x - \varepsilon ) - \varepsilon < \psi_{K} (x) \\
&\label{eq:7} 
\psi_{K} ( x - \varepsilon ) - \varepsilon < \varphi_{K} (x).
\end{align}
Furthermore, \eqref{eq:6}, \eqref{eq:7} and the definitions of $\varphi_{K}$ and $\psi_{K}$ imply that
\begin{equation}
\label{awkward} 
\varphi_{K} ( x - \varepsilon ) - \varepsilon < \psi (x)
\quad
\text{and}
\quad
\psi_{K} ( x - \varepsilon ) - \varepsilon < \varphi (x)
\end{equation}
hold for all $\varepsilon \in \mathcal{L}$ and $x \leq L$.

On the other hand, if $x > L$ then from the definition of $L$, we see that
$
K < \min \{ \varphi (x), \psi (x) \}
$.
Therefore for each $\varepsilon\in\mathcal{L}$, we have
$$
\psi_{K} ( x-\varepsilon ) - \varepsilon
<
K 
< \varphi (x)
$$
and
$$
\varphi_{K} ( x-\varepsilon ) - \varepsilon
<
K 
< \psi (x).
$$
Hence the inequalities \eqref{awkward} hold 
for all $\varepsilon\in\mathcal{L}$ and $x\ge 0$.

Now we are going to prove 
$\mathcal{L} \subset \mathcal{R}$,
where
$$
\mathcal{R}
:=
\Big\{
	\varepsilon > 0 :
	\text{$
	\begin{array}{c}
	\varphi^{-1} ( x-\varepsilon ) - \varepsilon
	\leq
	\psi^{-1} (x) \\
	\psi^{-1} ( x-\varepsilon ) - \varepsilon
	\leq
	\varphi^{-1} (x)
	\end{array}
	$
	for all $x \leq K$
	}
\Big\}.
$$

Let $\varepsilon\in\mathcal{L}$.
Since $\varphi^{-1}$ is non-negative, \eqref{awkward} implies that
$$
\psi
\big(
	\varphi^{-1} ( x ) + \varepsilon
\big) 
>
\varphi_{K}
\big(
	\varphi^{-1} ( x )
\big)
- \varepsilon
$$
holds for all $x\geq 0$. 
Furthermore, since $\varphi^{-1} ( x ) = \varphi_{K}^{-1} ( x )$ for all $x\leq K$, from \eqref{Ineq:01} we have
$$
\varphi_{K}
\big(
	\varphi^{-1} ( x )
\big)
- \varepsilon
\geq
x - \varepsilon
$$
for any $x\leq K$.
Therefore we have
\begin{equation}
\begin{split}
\label{eq:8} 
\psi^{-1} ( x - \varepsilon )
-
\varepsilon
\leq
\varphi^{-1} ( x )
\quad
\text{for all $x \leq K$.}
\end{split}
\end{equation}
Similarly, we also obtain that
\begin{equation}
\begin{split}
\label{eq:9} 
\varphi^{-1} ( x - \varepsilon )
-
\varepsilon
\leq
\psi^{-1} ( x )
\quad
\text{for all $x \leq K$.}
\end{split}
\end{equation}
Now from \eqref{eq:8}, \eqref{eq:9} and the definition of $\mathcal{R}$, 
we have that $\varepsilon \in \mathcal{R}$.
\end{proof} 

Now we are ready to prove Proposition \ref{inversion_reg}.

\begin{proof}[Proof of Proposition \ref{inversion_reg}] 
From
Proposition \ref{LevKol},
Proposition \ref{KolLev}
and
Proposition \ref{RegInv},
we obtain that
$$
\sup_{0 \leq x \leq K}
\vert \varphi^{-1} (x) - \psi^{-1} (x) \vert
\leq
\left(
	1 + \sup_{0 \leq x \leq K} \vert ( \varphi^{-1} )^{\prime} (x) \vert
\right)
\sup_{0 \leq x \leq L}
\vert
	\varphi_{K} (x) - \psi_{K} (x)
\vert ,
$$
where
$
L = \max \{ \varphi^{-1} (K), \psi^{-1} (K) \}
$
and
$f_{K} = \min \{ f, K \}$
for $f \in \mathcal{T}$.

However, since $\varphi$ and $\psi$ are non-decreasing,
we see that
$$
\vert
	\varphi_{K} (x) - \psi_{K} (x)
\vert
\leq
\vert
	\varphi_{K} (m) - \psi_{K} (m)
\vert
\leq
\vert
	\varphi (m) - \psi (m)
\vert
$$
for all
$
x \in ( m, L ]
$,
where
$m:=\min\{ \varphi^{-1}(K), \psi^{-1}(K) \}$.
Hence the proof is completed.
\end{proof} 

\subsection{Convergence implication for strong error}
\label{CISE} 

Assume the conditions (i)-(iv) in the subsection \ref{SFAB}.
Let
$\{ X_{s,t} \}_{s \leq t}$
and
$\{ X_{k,l}^{n} \}_{k\leq l}$
be the $(\sigma , b, w)$-stochastic flow
with the reflecting barrier at zero
and its Euler-Maruyama approximation defined in
Definition \ref{EM-AbsFl}, respectively.

\begin{Lem} 
\label{abs_conti} 
If $\sigma$ and $b$ are smooth then for each $s \leq t$, with probability one, 
the mapping
$X_{s, t}^{*} : [0, +\infty ) \to [0,+\infty )$
is absolutely continuous and for almost all $x \geq 0$, $( X_{s, t}^{*} )^{\prime} (x) \geq 0$. 
\end{Lem} 

\begin{proof} 
Let
$
x_{0} := \sup \{ x \geq 0 : X_{s,t}^{*} (x) = X_{s,t}^{*} (0) \}
$.
Since $\sigma$ and $b$ are smooth,
the mapping
$X_{s,t}^{*} \vert_{(x_{0}, +\infty )}$
is a diffeomorphism onto its image almost surely.
Therefore we have
\begin{equation*}
\begin{split}
X_{s,t}^{*} (x) - X_{s,t}^{*} (x_{0}+)
=
\int_{x_{0}}^{x}
( X_{s,t}^{*} )^{\prime} (y)
\mathrm{d}y
\end{split}
\end{equation*}
for any $x > x_{0}$.

On the other hand, from the definition of $x_{0}$,
we have
$
X_{s,t}^{*} \vert_{[0, x_{0})} (x) \equiv X_{s,t}^{*} (0)
$.
Thus we have
\begin{equation*}
\begin{split}
X_{s,t}^{*} (x_{0}-) - X_{s,t}^{*} (0) = 0 .
\end{split}
\end{equation*}
Since
$
X_{s,t}^{*}: [0, +\infty ) \to [0, +\infty )
$
is continuous,
we have
$X_{s,t}^{*} (x_{0}-) = X_{s,t}^{*} (x_{0}+)$
and hence
\begin{equation*}
\begin{split}
X_{s,t}^{*} (x) - X_{s,t}^{*} (0)
=
\int_{x_{0}}^{x}
( X_{s,t}^{*} )^{\prime} (y)
1_{ \{ y > x_{0} \} }
\mathrm{d}y
\end{split}
\end{equation*}
for any $x \geq 0$.
Hence $X_{s,t}^{*}$ is absolute continuous almost surely.

The non-negativity of $(X_{s,t}^{*})^{\prime}$
is obvious because $X_{s,t}^{*}$ is non-decreasing.
\end{proof} 

\begin{Rm} 
The differentiability or the derivative of stochastic flow generated by
stochastic differential equation with reflection
in multi-dimensional spaces are investigated by several authors,
e.g.,
Andres \cite{An},
Burdzy \cite{Bu},
Deuschel-Zambotti \cite{DZ},
Pilipenko \cite{Pi04, Pi05, Pi06a, Pi06b, Pi13},
and so on.
In particular, Pilipenko \cite{Pi13} found a stochastic equation
in which the first derivative of reflecting stochastic flow evolves.
This might be suited for our use.
Our situation is, however, rather simpler than theirs
because of one-dimensionality.
Hence we will not employ it.
\end{Rm} 

From Proposition \ref{inversion_reg}, Lemma \ref{abs_conti}
and stationarity of
$\{ X_{s,t} \}_{s \leq t}$
and
$\{ X_{k,l}^{n} \}_{k \leq l}$,
we obtain the following estimate.

\begin{Thm} 
\label{str-err} 
If $\sigma$ and $b$ are smooth then for each $K > 0$, we have
\begin{equation*}
\begin{split}
&
\mathbf{E}
\left[
\sup_{0 \leq x \leq K}
\vert X_{0,T}^{*} (x) - X_{0,n}^{n*} (x) \vert
\right] \\
&\leq
\mathbf{E}
\left[
\left(
	1 + \sup_{0 \leq x \leq K} ( X_{-T,0}^{*} )^{\prime} (x)
\right)
\sup_{ 0 \leq x \leq m }
\vert
	X_{0,T} (x) - X_{0,n}^{n} (x)
\vert
\right] ,
\end{split}
\end{equation*}
where
$
m:= \min \{ X_{-T,0}^{*} (K), X_{-n,0}^{n*} (K) \}
$.
\end{Thm} 

\begin{Cor} 
In addition to the assumptions in Theorem \ref{str-err},
suppose that $\sigma$ is a constant and $b^{\prime}$ is bounded.
Then for any $K > 0$, 
\begin{equation*}
\begin{split}
&\mathbf{E}
\left[
\sup_{0 \leq x \leq K}
\vert X_{0,T}^{*} (x) - X_{0,n}^{n*} (x) \vert
\right]\\
&\leq
\left( 1 + \mathrm{e}^{ \| b' \|_{\infty}T } \right)
\mathbf{E}
\left[
\sup_{ 0 \leq x \leq m }
\vert
	X_{0,T} (x) - X_{0,n}^{n} (x)
\vert
\right]
\end{split}
\end{equation*}
holds, where
$
m:= \min \{ X_{-T,0}^{*} (K), X_{-n,0}^{n*} (K) \}
$.
\end{Cor} 
\begin{proof} 
From stationarity of $\{X_{s,t}^{*}\}_{s\leq t}$,
it is enough to show that
$
\sup_{0 \leq x \leq K}
( X_{0,T}^{*} )^{\prime} (x)
$
is bounded by $1 + \mathrm{e}^{\| b' \|_{\infty}}$.

Recall that from Theorem \ref{Principle}, 
we have that for any $x \geq 0$,
\begin{equation*}
X_{0,t}^{*} (x)
=
x + \sigma \widehat{w}^{(0)} (t)
-
\int_{0}^{t} b( X_{0,s}^{*} (x) ) \mathrm{d}s
+ \phi_{t} (x),
\end{equation*}
where $(\phi_{t}(x))_{t \geq 0}$ is
an $(\mathcal{F}_{0,t}^{\widehat{w}})_{t \geq 0}$-adapted
stochastic process such that the path $t\mapsto \phi_{t}(x)$ is
non-decreasing, continuous, $\phi_{0}(x) = 0$
and increases at $t$ only when $X_{0,t}^{*} (x) = 0$.

On the other hand, since
$X_{0,t}^{*}: [0, +\infty ) \to [0, +\infty )$
is non-decreasing and
the pathwise uniqueness holds for (\ref{SKR^}),
$\phi_{t}(x)$
must be non increasing in $x$.
In particular, $\phi_{t}(x)$ is
differentiable in $x$ and $\phi_{t}^{\prime} (x) \leq 0$ 
for almost all $x \geq 0$, almost surely.
Therefore we have
\begin{equation*}
\begin{split}
( X_{0,t}^{*} )^{\prime} (x)
&=
1
-
\int_{0}^{t}
( X_{0,s}^{*} )^{\prime} (x)
b^{\prime} ( X_{0,s}^{*} (x) ) \mathrm{d}s
+ \phi_{t}^{\prime} (x) \\
&\leq
1
-
\int_{0}^{t}
( X_{0,s}^{*} )^{\prime} (x)
b^{\prime} ( X_{0,s}^{*} (x) ) \mathrm{d}s.
\end{split}
\end{equation*}
Hence we have
\begin{equation*}
\begin{split}
( X_{0,t}^{*} )^{\prime} (x)
\leq
1
+
\int_{0}^{t}
\vert b^{\prime} ( X_{0,s}^{*} (x) ) \vert
( X_{0,s}^{*} )^{\prime} (x)
\mathrm{d}s.
\end{split}
\end{equation*}
Now, Gronwall's lemma implies that
$
\sup_{0 \leq x \leq K}
(X_{0,T}^{*})^{\prime} (x)
\leq
\mathrm{e}^{ \Vert b^{\prime} \Vert_{\infty} T }
$,
where $\Vert b^{\prime} \Vert_{\infty} := \sup_{x\geq 0}\vert b^{\prime}(x) \vert$
and hence we obtain the result.
\end{proof} 

\appendix
\section{\ }
\subsection{Some inequalities for elements of $\mathcal{T}$}
\label{Appdx} 
Recall that $\mathcal{T}$ is the set of all non-decreasing and right-continuous functions
$\varphi : [0,+\infty ) \to [0,+\infty )$.

For $\varphi\in\mathcal{T}$,
$
\varphi^{-1}
$
denote the right-continuous inverse of $\varphi$, that is, 
$
\varphi^{-1} (x)
:=
\inf \{ y \geq 0 : \varphi (y) > x \}
$,
$x \geq 0$.
Then we have the following proposition.
\begin{Prop} 
\label{app-prop} 
If $\varphi$ and $\psi \in \mathcal{T}$
satisfy
\begin{equation}
\label{app-ass} 
\lim_{x \to +\infty} \varphi (x)
= \lim_{x \to +\infty} \psi (x)
= +\infty ,
\end{equation}
then $\varphi^{-1}$ and $\psi^{-1}$ belong to $\mathcal{T}$. 
Furthermore for each $z \geq 0$, we have the following.
\begin{itemize}
\item[(i)]
For each $\varepsilon > 0$,
$
\varphi ( \varphi^{-1}(z) + \varepsilon ) > z
$.

\vspace{2mm}
\item[(ii)]
$
( \varphi \circ \psi )^{-1} (z)
\leq
\psi^{-1} ( \varphi^{-1} (z))
$.

\vspace{2mm}
\item[(iii)]
\begin{align*}
\varphi^{-1} (z)
\leq
\psi ( ( \varphi \circ \psi )^{-1} (z) )
\leq
\psi ( \psi^{-1} ( \varphi^{-1} (z)) )
\end{align*}
and
\begin{align*}
\psi ( \psi^{-1} ( \varphi^{-1} (z)) -  )
\leq
\varphi^{-1} (z).
\end{align*}

\vspace{2mm}
\item[(iv)]
If $\varphi$ is left-continuous at $\varphi^{-1}(z)$
then
$
( \varphi \circ \psi )^{-1} (z)
=
\psi^{-1} ( \varphi^{-1} (z))
$.

\end{itemize}
\end{Prop} 
\begin{proof} 
Under the condition \eqref{app-ass},
it is easy to check that
$\varphi^{-1}$ and $\psi^{-1}$ belong to $\mathcal{T}$.
We now prove inequalities (i)-(iv).

We first prove (i).
Assume that there exists some $\varepsilon >0$ such that
$
\varphi ( \varphi^{-1}(z) + \varepsilon ) \leq z
$.
Then it must be
$\varphi^{-1} (z) + \varepsilon \leq \varphi^{-1}(z)$,
but this is impossible.
Therefore for any $\varepsilon>0$, we have
$
\varphi ( \varphi^{-1}(z) + \varepsilon ) > z
$.

For (ii), by the definition, there exists some decreasing sequence
$(y_{n})_{n=1}^{\infty}$
such that it converges to $\varphi^{-1}(z)$
and satisfies
$\varphi (y_{n}) > z$ for any $n$.
Then from the definition of $\psi^{-1}$ and monotonicity of $\varphi$, we have
$$
( \varphi \circ \psi )
( \psi^{-1} (y_{n}) )
\geq
\varphi ( y_{n} ) > z.
$$
Therefore we have
$
( \varphi \circ \psi )^{-1} (z)
\leq
\psi^{-1} (y_{n})
$
for any $n$.
Now by letting $n$ tend to infinity, the right-continuity of $\psi^{-1}$
yields our desired inequality.

Now we turn to prove (iii).
Let $\varepsilon$ be an arbitrary positive number.
From the definition of $\psi^{-1} (y)$, we have
$$
\psi ( \psi^{-1}(y) - \varepsilon )
\leq
y
$$
for any $y \geq 0$.
Hence by letting $\varepsilon$ tend to zero, the second inequality of (iii) is obtained.

On the other hand, from (i), we have that
$$
z
<
\varphi
\big(
	\psi
	(
	(\varphi \circ \psi)^{-1} (z)
	+
	\varepsilon
	)
\big).
$$
Hence we have
$$
\varphi^{-1} (z)
\leq
\psi
(
	(\varphi \circ \psi)^{-1} (z)
	+
	\varepsilon
).
$$
Now by letting $n$ tend to infinity, the right-continuity of $\psi$ yields
\begin{equation*}
\varphi^{-1} (z)
\leq
\psi
(
(\varphi \circ \psi)^{-1} (z)
).
\end{equation*}
Furthermore, from (ii), we have
\begin{equation*}
\psi
(
(\varphi \circ \psi)^{-1} (z)
)
\leq
\psi
(
	\psi^{-1} ( \varphi^{-1} (z) )
).
\end{equation*}
From these inequalities we have the first inequality of (iii).

We are going to prove (iv).
We first show that
\begin{equation}
\label{psi>phi} 
\varphi^{-1} ( z )
<
\psi
(
(\varphi \circ \psi)^{-1} (z) + \varepsilon
)
\quad
\text{for all $\varepsilon > 0$.}
\end{equation}
To do this, we consider the contrary:
$$
\psi
(
(\varphi \circ \psi)^{-1} (z) + \varepsilon_{0}
)
\leq
\varphi^{-1} ( z )
\quad
\text{for some $\varepsilon_{0} > 0$.}
$$
Since $\varphi$ is increasing and continuous at $\varphi^{-1}(z)$, we have
$$
( \varphi \circ \psi )
(
(\varphi \circ \psi)^{-1} (z) + \varepsilon_{0}
)
\leq
z.
$$
However, this contradicts to (i).
Therefore \eqref{psi>phi} holds.

Now \eqref{psi>phi} and the definition of $\psi^{-1}$ imply that
$$
\psi^{-1} ( \varphi^{-1} (z) )
\leq
(\varphi \circ \psi)^{-1} (z) + \varepsilon
$$
holds for all $\varepsilon >0$.
Now by letting $\varepsilon$ tend to zero and applying (ii), we obtain the result.
\end{proof} 
\subsection{Proof of Theorem \ref{Principle}}
\label{Appdx2} 
We give below a proof of Theorem \ref{Principle}
in the case where
$\sigma \vert_{(0,+\infty)}$
and
$b \vert_{(0,+\infty)}$
are bounded.
Note that since $\sigma$ and $b$ are bounded, $\widehat{\sigma}$ and $\widehat{b}$ are also bounded.

The proof will be completed when we succeeded to show that
for each $s \in \mathbb{R}$ and $x \in [0, +\infty )$,
$\widehat{X}_{t}: = (X_{s,t}^{-})^{*} (x)$, $t \geq s$
solves \eqref{SKR^}
with respect to $\widehat{w}^{(s)}$,
such that $\widehat{X}_{s}=x$.
For simplicity, we set $s=0$:
$\widehat{X}_{t} = ( X_{0,t}^{-} )^{*} (x)$,
$t \geq 0$.
One can imitate easily the arguments below
to prove the case for general $s \in \mathbb{R}$.

We first prove the continuity of $\widehat{X}_{t}$ in $t$.
\begin{Lem} 
\label{Continuity} 
The function $t \mapsto \widehat{X}_{t}$ is continuous almost surely.
\end{Lem} 
\begin{proof} 
Take $t > 0$ arbitrary.
On the event
$\{ \widehat{X}_{t} > 0 \}$,
$( \widehat{X}_{s} )_{t-\delta < s < t + \delta}$
with sufficiently small $\delta > 0$,
evolves as
$$
\mathrm{d} \widehat{X}_{s}
=
\widehat{\sigma} ( \widehat{X}_{s} ) \mathrm{d} \widehat{w}^{(0)}(s)
+
\widehat{b} ( \widehat{X}_{s} ) \mathrm{d} s ,
$$
where $\widehat{\sigma} = \sigma$
and
$\widehat{b} = \sigma \sigma^{\prime} - b$
(see \cite[Chapter V-Lemma 2.2]{IW}).
Therefore, the continuity at $t$ is obvious on this set.

Now we show the continuity at $t$ on the event
$\{ \widehat{X}_{t} = 0 \}$.
Take $\varepsilon > 0$ arbitrary.

We first show the right-continuity at $t$.
To do this, we shall make the following observation:
\begin{equation}
\label{full-prob} 
\mathbf{P}
\left(
	\begin{array}{l}
	\text{there exists $\delta > 0$ such that} \\
	\text{$
	X_{ -t - \delta^{\prime} , -t }^{-} ( \varepsilon / 2 ) > 0
	$ for each $\delta^{\prime} \in (0, \delta)$}
	\end{array}
\right) = 1 .
\end{equation}
In fact, if we assume the contrary:
$$
p
:=
\mathbf{P}
\left(
	\begin{array}{l}
	\text{for each $\delta > 0$, there exists
	$\delta^{\prime} \in (0, \delta)$} \\
	\text{such that
	$
	X_{ -t - \delta^{\prime} , -t }^{-} ( \varepsilon / 2 ) = 0
	$}
	\end{array}
\right) > 0 ,
$$
then we see that
$$
\limsup_{\delta \downarrow 0}
\mathbf{P} ( X_{ -t - \delta , -t }^{-} ( \varepsilon /2 ) = 0 )
\geq p > 0.
$$
However the stationarity and
stochastic continuity for $\{ X_{s,t}^{-} \}_{s \leq t}$ 
imply that
$
\mathbf{P} ( X_{ -t - \delta , -t }^{-} ( \varepsilon / 2 ) = 0 )
=
\mathbf{P} ( X_{0, \delta}^{-} ( \varepsilon / 2 ) = 0 )
$
converges to zero as $\delta$ tends to zero.
This is a contradiction.

Therefore, by (\ref{full-prob}), we can take a
$\delta > 0$ such that 
$X_{-t-\delta^{\prime}, -t}^{-} ( \varepsilon / 2 ) > 0$
for each $\delta^{\prime} \in (0, \delta )$,
and hence we have
$( X_{-t-\delta^{\prime}, -t}^{-} )^{-1} (0) \leq \varepsilon / 2$.
Now, with noting that
$
( X_{-t-\delta^{\prime}, 0}^{-1} ) (x) 
=
( X_{-t-\delta^{\prime}, -t}^{-})^{-1}( (X_{-t, 0}^{-})^{-1}(x) )
$
and
$(X_{-t, 0}^{-})^{-1}(x) = \widehat{X}_{t} = 0$,
we have that for each $\delta^{\prime} \in (0, \delta )$,
\begin{equation*}
\begin{split}
\vert
	\widehat{X}_{t+\delta^{\prime}}
	-
	\widehat{X}_{t}
\vert
&=
( X_{-t-\delta^{\prime}, 0}^{-1} ) (x) \\
&=
( X_{-t-\delta^{\prime}, -t}^{-})^{-1} (0)\\
&\leq \varepsilon /2.
\end{split}
\end{equation*}
This prove the right-continuity at $t$.

Now we turn to prove the left-continuity at $t$.
Since $\widehat{X}_{t}=0$,
by the definition,
there exists
$y > 0$ such that $y < \varepsilon /2$
and
$X_{-t,0}^{-}(y) > x$.
Then the continuity of the map 
$[-t, +\infty ) \ni u \mapsto X_{-t, u}^{-}(y)$
implies that there exists $\delta > 0$ such that
$
\vert X_{-t, -t + \delta^{\prime} }^{-} (y) - y \vert < \varepsilon /2
$
for any $\delta^{\prime} \in [0, \delta )$.
Furthermore, we have that
$
X_{-t, -t + \delta^{\prime} }^{-} (y)
\geq
( X_{-t + \delta^{\prime}, 0 }^{-} )^{-1} (x)
$
because
$
X_{-t + \delta^{\prime}, 0 }^{-}
\big(
	X_{-t, -t + \delta^{\prime} }^{-} (y)
\big)
=
X_{-t, 0 }^{-} ( y ) > x
$.
Hence for any $\delta^{\prime} \in [0, \delta )$, we obtain that
\begin{equation*}
\begin{split}
\vert \widehat{X}_{t} - \widehat{X}_{t-\delta^{\prime}} \vert
&=
( X_{-t+\delta^{\prime}, 0}^{-} )^{-1} (x) \\
&\leq
X_{-t, -t + \delta^{\prime} }^{-} (y)\\
&\leq
\vert X_{-t, -t + \delta^{\prime} }^{-} (y) - y \vert
+ y\\
&< \varepsilon .
\end{split}
\end{equation*}
\end{proof} 

For each $\varepsilon > 0$,
we define a sequence $(\tau_{n}^{\varepsilon})_{n=0}^{\infty}$
by
$\tau_{0}^{\varepsilon} := 0$,
and
\begin{equation*}
\tau_{n}^{\varepsilon}
:=
\inf
\{
	t > \tau_{n-1}^{\varepsilon} :
	\vert \widehat{X}_{t} - \widehat{X}_{\tau_{n-1}^{\varepsilon}} \vert > \varepsilon
\}
\end{equation*}
for $n \geq 1$.
For a convention, we define $\inf \emptyset := +\infty$.
Here we make the following observations.

(i) If
$
\widehat{X}_{\tau_{n-1}^{\varepsilon}} \leq \varepsilon
$
then
$
\widehat{X}_{\tau_{n}^{\varepsilon}} > \widehat{X}_{\tau_{n-1}^{\varepsilon}}
$
(this is why we have used
``$> \varepsilon$"
rather than
``$= \varepsilon$"
in the definition of $\tau_{n}^{\varepsilon}$ ).

(ii) Assume that if
$
\widehat{X}_{\tau_{n-1}^{\varepsilon}} > \varepsilon
$
then $\widehat{X}_{t} >0$ holds for all
$
t \in [\tau_{n-1}^{\varepsilon}, \tau_{n}^{\varepsilon}]
$.
Then 
$
\widehat{X}_{t}
=
( X_{ -t, -\tau_{n-1}^{\varepsilon} }^{-} )^{-1}
\big(
	\widehat{X}_{ \tau_{n-1}^{\varepsilon} }
\big)
$
evolves as
\begin{equation}
\label{interior} 
\widehat{X}_{t}
=
\widehat{X}_{\tau_{n-1}^{\varepsilon}}
+
\int_{\tau_{n-1}^{\varepsilon}}^{t}
\widehat{\sigma} ( \widehat{X}_{u} )
\mathrm{d} \widehat{w}^{(0)}(u)
+
\int_{\tau_{n-1}^{\varepsilon}}^{t}
\widehat{b} ( \widehat{X}_{u} )
\mathrm{d} u
\end{equation}
for $t \in [\tau_{n-1}^{\varepsilon}, \tau_{n}^{\varepsilon}]$
(see \cite[Chapter V-Lemma 2.2]{IW}).

(iii)
If (i) occurs, then it must be
$\widehat{X}_{\tau_{n}^{\varepsilon}} > \varepsilon$,
so that
$
\sharp \{ n:  \widehat{X}_{\tau_{n}^{\varepsilon}} \leq \varepsilon\}
\leq
\sharp \{ n:  \widehat{X}_{\tau_{n}^{\varepsilon}} > \varepsilon \}
$.
Therefore, the event
$\{ \widehat{X}_{\tau_{n}^{\varepsilon}} > \varepsilon \}$
occurs for infinitely many $n$ on the event
$
\{
	\text{$\tau_{n}^{\varepsilon} <+\infty $ for all $n$}
\}
$.
\begin{Lem} 
The sequence of random variables
$(\tau_{n}^{\varepsilon})_{n=0}^{\infty}$
tends to infinity
almost surely.
\end{Lem} 
\begin{proof} 
Let
$
K :=
\sup_{y > 0}
( \vert \widehat{\sigma} (y) \vert + \vert \widehat{b} (y) \vert )
< +\infty
$.
From the observation (ii) and a well-known estimate
(see \cite[Chapter V-Lemma 10.5]{IW}),
if $n^{-1} \in ( 0, \varepsilon /{2K} ]$
we have that
\begin{equation*}
\begin{split}
\mathbf{P}
(
	\tau_{n}^{\varepsilon} - \tau_{n-1}^{\varepsilon}
	< n^{-1}
	\vert
	\mathcal{F}_{-\infty , \tau_{n-1}^{\varepsilon}}^{\widehat{w}}
)
\leq
\frac{4}{\sqrt{\pi \varepsilon}}
\exp
\left(
	- \frac{ n \varepsilon^{2} }{ 8 K }
\right)
\quad
\end{split}
\end{equation*}
on the event
$
\{ \widehat{X}_{\tau_{n-1}^{\varepsilon}} > \varepsilon \}
$.
From this inequality, we obtain that
\begin{equation*}
\begin{split}
&
\sum_{n=1}^{\infty}
\mathbf{P}
(
	\tau_{n}^{\varepsilon} - \tau_{n-1}^{\varepsilon}
	\geq n^{-1}
	\vert
	\mathcal{F}_{-\infty , \tau_{n-1}^{\varepsilon}}^{\widehat{w}}
) \\
&\geq
\sum_{n=1}^{\infty}
\mathbf{P}
(
	\tau_{n}^{\varepsilon} - \tau_{n-1}^{\varepsilon}
	\geq n^{-1}
	\vert
	\mathcal{F}_{-\infty , \tau_{n-1}^{\varepsilon}}^{\widehat{w}}
)
1_{\{ \widehat{X}_{\tau_{n-1}^{\varepsilon}} > \varepsilon \}} \\
&\geq
\sum_{\text{$n$: $n^{-1} \in (0, \varepsilon / 2K)$}}
\left(
	1
	-
	\frac{4}{\sqrt{\pi \varepsilon}}
	\exp
	\left(
		- \frac{ n \varepsilon^{2} }{ 8 K }
	\right)
\right)
1_{\{ \widehat{X}_{\tau_{n-1}^{\varepsilon}} > \varepsilon \}}.
\end{split}
\end{equation*}
The last quantity diverges because of the observation (iii).
Therefore, by L\'evy's extension of the Borel-Cantelli lemmas
(see \cite[Chapter 12-Theorem 12.15]{Wi}),
we can conclude that
$
\tau_{n}^{\varepsilon} - \tau_{n-1}^{\varepsilon} \geq n^{-1}
$
infinitely often.
This implies that
$
\tau_{n}^{\varepsilon}
$
tends to infinity almost surely.
\end{proof} 

Bear in mind the above observations. 
If we put
$$
\phi^{\varepsilon} (t)
:=
\sum_{n=1}^{\infty}
(
	\widehat{X}_{t \wedge \tau_{n}^{\varepsilon}}
	-
	\widehat{X}_{t \wedge \tau_{n-1}^{\varepsilon}}
)
1_{ \{ \widehat{X}_{\tau_{n-1}^{\varepsilon}} \leq \varepsilon \} }
$$
then we can write
\begin{equation*}
\begin{split}
\widehat{X}_{t} - x
&=
\sum_{n=1}^{\infty}
(
	\widehat{X}_{t \wedge \tau_{n}^{\varepsilon}}
	-
	\widehat{X}_{t \wedge \tau_{n-1}^{\varepsilon}}
)
1_{ \{ \widehat{X}_{\tau_{n-1}^{\varepsilon}} > \varepsilon \} }
+
\phi^{\varepsilon} (t) \\
&=
\sum_{n=1}^{\infty}
\left(
	\int_{t \wedge \tau_{n-1}^{\varepsilon}}^{t \wedge \tau_{n}^{\varepsilon}}
	\widehat{\sigma} ( \widehat{X}_{u} )
	\mathrm{d} \widehat{w}^{(0)}(u)
	+
	\int_{t \wedge \tau_{n-1}^{\varepsilon}}^{t \wedge \tau_{n}^{\varepsilon}}
	\widehat{b} ( \widehat{X}_{u} )
	\mathrm{d} u
\right)
1_{ \{ \widehat{X}_{\tau_{n-1}^{\varepsilon}} > \varepsilon \} }
+
\phi^{\varepsilon} (t) .
\end{split}
\end{equation*}
Since
$
1_{ \{ \widehat{X}_{\tau_{n-1}^{\varepsilon}} > \varepsilon \} }
=
1 - 1_{ \{ \widehat{X}_{\tau_{n-1}^{\varepsilon}} \leq \varepsilon \} }
$,
we obtain
\begin{equation}
\label{eq:1} 
\begin{split}
\widehat{X}_{t} - x
&=
\int_{0}^{t}
\widehat{\sigma} ( \widehat{X}_{u} )
1_{\{ \widehat{X}_{u} > 0 \}}
\mathrm{d} \widehat{w}^{(0)}(u)
+
\int_{0}^{t}
\widehat{b} ( \widehat{X}_{u} )
1_{\{ \widehat{X}_{u} > 0 \}}
\mathrm{d} u
+
\phi^{\varepsilon} (t)
+ R_{t}^{\varepsilon},
\end{split}
\end{equation}
where
$$
R_{t}^{\varepsilon}
:=
\sum_{n=1}^{\infty}
\Big\{
	\int_{t \wedge \tau_{n-1}^{\varepsilon}}^{t \wedge \tau_{n}^{\varepsilon}}
	\widehat{\sigma} ( \widehat{X}_{u} )
	\mathrm{d} \widehat{w}^{(0)}(u)
	+
	\int_{t \wedge \tau_{n-1}^{\varepsilon}}^{t \wedge \tau_{n}^{\varepsilon}}
	\widehat{b} ( \widehat{X}_{u} )
	\mathrm{d} u
\Big\}
1_{ \{ \widehat{X}_{\tau_{n-1}^{\varepsilon}} \leq \varepsilon \} }.
$$
Now we shall prove that $R_{t}^{\varepsilon}$ converges to zero in probability,
as $\varepsilon$ tends to zero.
The It\^{o}-isometry and Fubini's theorem imply that
\begin{equation*}
\begin{split}
&
\mathbf{E}
\left[
	\left(
		\sum_{n=1}^{\infty}
		\int_{t \wedge \tau_{n-1}^{\varepsilon}}^{t \wedge \tau_{n}^{\varepsilon}}
		\widehat{\sigma} ( \widehat{X}_{u} )
		\mathrm{d} \widehat{w}^{(0)}(u)
		1_{ \{ \widehat{X}_{\tau_{n-1}^{\varepsilon}} \leq \varepsilon \} }
	\right)^{2}
\right] \\
&\leq
\mathbf{E}
\big[
	\int_{0}^{t}
	\widehat{\sigma} ( \widehat{X}_{u} )^{2}
	1_{ \{ \widehat{X}_{u} \leq 2\varepsilon \} }
	\mathrm{d} u
\big]\\
&=
\int_{0}^{t}
\mathbf{E}
\big[
	\widehat{\sigma} ( \widehat{X}_{u} )^{2}
	1_{ \{ \widehat{X}_{u} \leq 2\varepsilon \} }
\big]
\mathrm{d} u.
\end{split}
\end{equation*}
Since $\widehat\sigma$ is bounded, from the monotone convergence theorem, 
the last term converges to
$
\int_{0}^{t}
\mathbf{E}
[
	\widehat{\sigma} ( \widehat{X}_{u} )^{2}
	1_{ \{ \widehat{X}_{u} = 0 \} }
]
\mathrm{d} u
$
as $\varepsilon$ tends to zero.
The next lemma tells us this quantity is zero.
\begin{Lem} 
\label{no_time} 
For each $t \geq 0$,
we have
$
\mathbf{P}
( \widehat{X}_{t} = 0 )
= 0
$.
\end{Lem} 
\begin{proof} 
Suppose the contrary.
Then  by the definition, we can choose a decreasing sequence
$y_{n} \geq 0$ such that
$y_{n}$ converges to zero and $X_{-t,0}^{-} (y_{n}) > 0$
with a positive probability.
However this contradicts to Lemma \ref{excit}.
\end{proof} 

From this lemma and the previous inequality, we conclude that
\begin{equation*}
\begin{split}
\lim_{\varepsilon \downarrow 0}
\mathbf{E}
\left[
	\left(
		\sum_{n=1}^{\infty}
		\int_{t \wedge \tau_{n-1}^{\varepsilon}}^{t \wedge \tau_{n}^{\varepsilon}}
		\widehat{\sigma} ( \widehat{X}_{u} )
		\mathrm{d} \widehat{w}^{(0)}(u)
		1_{ \{ \widehat{X}_{\tau_{n-1}^{\varepsilon}} \leq \varepsilon \} }
	\right)^{2}
\right]
= 0.
\end{split}
\end{equation*}
Similarly we have
\begin{equation*}
\begin{split}
\lim_{\varepsilon \downarrow 0}
\mathbf{E}
\left[
		\sum_{n=1}^{\infty}
		\int_{t \wedge \tau_{n-1}^{\varepsilon}}^{t \wedge \tau_{n}^{\varepsilon}}
		\vert \widehat{b} ( \widehat{X}_{u} ) \vert
		\mathrm{d} u
		1_{ \{ \widehat{X}_{\tau_{n-1}^{\varepsilon}} \leq \varepsilon \} }
\right]
= 0,
\end{split}
\end{equation*}
so that $R_{\varepsilon}$ converges to zero
in probability as $\varepsilon$ tends to zero.
Due to this convergence and \eqref{eq:1}, we can define
$\phi (t):=\lim_{\varepsilon \to 0} \phi^{\varepsilon} (t)$.
From the construction of $\phi(t)$ and the observation (i),
it is clear that
$\phi (t)$ is a non-negative process and
increases only when $\widehat{X}_{t} = 0$.
Consequently, we have
\begin{equation*}
\begin{split}
\widehat{X}_{t}
&=
x
+
\int_{0}^{t}
\widehat{\sigma} ( \widehat{X}_{u} )
1_{\{ \widehat{X}_{u} > 0 \}}
\mathrm{d} \widehat{w}^{(0)}(u)
+
\int_{0}^{t}
\widehat{b} ( \widehat{X}_{u} )
1_{\{ \widehat{X}_{u} > 0 \}}
\mathrm{d} u
+
\phi (t).
\end{split}
\end{equation*}
In particular, from Lemma \ref{Continuity}, $\phi (t)$ is continuous.
The proof of Theorem \ref{Principle}
(in the case where
$\sigma \vert_{(0,+\infty)}$
and
$b \vert_{(0,+\infty)}$
are bounded)
ends.

\end{document}